%% file: positivity.tex
\documentclass[12pt]{amsart}
\usepackage{mathtools}
\usepackage{nccmath}
\usepackage{braket}
\usepackage{mathabx}
\usepackage[mathscr]{eucal}
\input{ams_header}
\usepackage{tikz}
\usepackage[english]{babel}

\newcommand{\I}{\mathbbm{1}} 
\newcommand{\ang}[1]{\langle #1 \rangle}
\newcommand{\norm}[1]{\lVert #1 \rVert} 
\newcommand{\wtd}{\widetilde}
\renewcommand{\abs}[1]{\lvert #1 \rvert}

\newcommand{\F}{\mathcal{F}}

\DeclareMathOperator{\HALT}{HALT}
\DeclareMathOperator{\coHALT}{coHALT}
\DeclareMathOperator{\PSPACE}{PSPACE}

\newcommand{\sgn}{\text{sgn}}

\title{Positivity is undecidable in tensor products of free algebras}

\author[Arthur Mehta]{Arthur Mehta$^{1,2}$}
\author[William Slofstra]{William Slofstra$^{3,4}$}
\author[Yuming Zhao]{Yuming Zhao$^{5}$}

\address[1]{Department of Mathematics and Statistics, University of Ottawa}
\address[2]{Nexus for Quantum Technologies, University of Ottawa}
\address[3]{Institute for Quantum Computing, University of Waterloo}
\address[4]{Department of Pure Mathematics, University of Waterloo}
\address[5]{QMATH, Department of Mathematical Sciences, University of Copenhagen}

\email{amehta2@uottawa.ca}
\email{weslofst@uwaterloo.ca}
\email{yuming@math.ku.dk}

\date{}
\begin{document}
\maketitle
\begin{abstract}
It is well known that an element of the algebra of noncommutative $*$-polynomials is positive in all $*$-representations if and only if it is a sum of squares. This provides an effective way to determine if a given $*$-polynomial is positive, by searching through sums of squares decompositions. We show that no such procedure exists for the tensor product of two noncommutative $*$-polynomial algebras: determining whether a $*$-polynomial of such an algebra is positive is coRE-hard. 
We also show that it is coRE-hard to determine whether a noncommutative $*$-polynomial is trace-positive. Our results hold if noncommutative $*$-polynomial algebras are replaced by other sufficiently free algebras such as group algebras of free groups or free products of cyclic groups.

\end{abstract}

\section{Introduction}

Recall that a self-adjoint bounded operator $T$ on a Hilbert space $H$ is
positive if $\langle v, T v \rangle \geq 0$ for all $v \in H$.  A self-adjoint
element $\alpha$ of a complex $*$-algebra $\mcA$ is said to be \emph{positive}
if $\pi(\alpha)$ is positive for all $*$-representations $\pi : \mcA \to \mcB(H)$.
Suppose that $\mcA$ is the complexification of a finitely generated $*$-algebra
$\mcA_{\Q}$ over $\Q$.  Elements of $\mcA_{\Q}$ can be represented as
noncommutative $*$-polynomials with rational coefficients, and hence we can work with them
concretely, for instance on a computer.
 
In many situations, we'd like to be able to decide whether an element $a \in
\mcA_{\Q}$ is positive in $\mcA$, and if it is, find a certificate of
positivity.  A prototypical example is $\mcA = \C \F_n$, the group algebra of
the free group $\F_n$, with $\mcA_{\Q} = \Q \mcF_n$. In this case, the
following two properties hold:
\begin{enumerate}[(a)]
    \item Every positive element $\alpha \in \mcA$ is a \emph{sum of (hermitian)
        squares (SOS)}, meaning that there are $b_1,\ldots,b_k \in \mcA$ such that
        $\alpha = \sum_{i=1}^k b_i^* b_i$. 
    \item A self-adjoint element $\alpha \in \mcA$ is positive if and only if 
        $\pi(\alpha)$ is positive in all finite-dimensional representations
        $\pi : \mcA \to \mcB(H)$.
\end{enumerate}
Property (a) is called the \emph{SOS property}; for
archimedean algebras (see the next section for a definition) such as $\C \F_n$,
this property is also called \emph{archimedean closed}. Since a sum of squares
is always positive, a sum of squares decomposition for an element $a$ can be
used as a certificate of positivity.  The SOS property states that every
positive element has such a certificate.  For archimedean algebras, property
(b) states that the enveloping $C^*$-algebra of $\mcA$ is \emph{residually
finite-dimensional (RFD)}, or equivalently that the quadratic module generated
by hermitian squares in $\mcA$ is RFD \cite{ANT18}. For convenience, we'll
shorten this and just say that $\mcA$ is RFD.\footnote{But note that this is
not equivalent to the standard definition of RFD for $*$-algebras, which says
that all non-zero elements are non-zero in some finite-dimensional
representation.} If $\mcA$ is RFD, then for every non-positive self-adjoint
element $\alpha \in \mcA$, there is a finite-dimensional representation $\phi$ such
that $\phi(\alpha)$ is not positive.  Thus if $\mcA_{\Q}$ is finitely presented, then
this representation can be used as a certificate of non-positivity.  
If $\mcA$ has the SOS property and is RFD, and $\mcA_{\Q}$ is
finitely presented, it is possible to decide whether a self-adjoint element $\alpha
\in \mcA_{\Q}$ is positive by simultaneously searching for a sum of squares
decomposition for $\alpha$, and for a finite-dimensional representation showing that
$\alpha$ is not positive. 

For the free group algebra, property (a) is sometimes attributed to an
unpublished result of Schm\"{u}dgen from the 1980s (see \cite{NT12,Ozawa13}),
and can also be found in \cite{McC01,HMP,BT07,NT12,Ozawa13}.  Property (b) is a
result of Choi \cite{Choi}. Both properties also hold for the $*$-algebra of
noncommutative $*$-polynomials $\C^*\ang{x_1,\ldots,x_n}$ by theorems of Helton
\cite{Hel02} and McCullough \cite{McC01}, and for the $*$-algebra of $n$
contractions by theorems of Helton-McCullough-Putinar \cite{HMP} or
Helton-Klep-McCullough \cite{HKM12}. For the group algebra of a virtually free
group such as the free product $\Z_m^{*n}$ of cyclic groups, a proof of
both properties is again attributed to Schm\"{u}dgen \cite{NT12}, with a
complete proof appearing in \cite{GonenCohen2023}.  That all these algebras are
RFD also follows from \cite{GM90} or \cite{EL92}. Thus it is possible to decide
positivity for all these algebras using the algorithm above, although actually
we can do better: the proofs of the SOS property above are effective, meaning
that they provide bounds on the degree of polynomials in the sum of squares
needed to show that an element is positive. Hence we can decide positivity by
searching just through SOS decompositions and stopping when the bound is
reached. 

In comparison, the commutative polynomial rings $\C[x_1,\ldots,x_n]$ famously
do not have the SOS property for $n \geq 2$. However, Artin's solution of
Hilbert's 17th problem states that every positive polynomial is a sum of
squares of rational functions \cite{Art27}.  More generally, a result of
Scheiderer \cite{Sch00} states that finitely generated connected commutative
$*$-algebras of dimension $\geq 3$, such as the group algebra $\C \Z^n$ for $n
\geq 3$, do not have the SOS property (provided they have a state). 
Certificates of positivity are provided for all finitely generated commutative
algebras by the Krivine-Stengle Positivstellensatz \cite{Kri64,Ste74}.
Effective versions of Artin's theorem and the Krivine-Stengle
Positivstellensatz are proven in \cite{LPR20}, and although the degree bounds
are too large to yield practical algorithms, this gives a way to decide
positivity in principle. For finitely generated commutative algebras,
positivity can also be decided with Tarski quantifier elimination or other
algorithms for the first-order theory of the reals. In fact, since
non-positivity lies in the existential theory of the reals, positivity can be
decided in PSPACE \cite{Can88}.  

If $\mcA$ is a finitely generated commutative archimedean $*$-algebra, then
Schm\"{u}dgen's Positivstellensatz states that every element which is strictly
positive on the state space of $\mcA$ is a sum of squares \cite{Sch91}. For all
archimedean $*$-algebras $\mcA$, Helton and McCullough's Positivstellensatz
\cite{HM04} says that $\alpha \in \mcA$ is positive if and only if $\alpha +
\eps$ is a sum of squares for all $\eps > 0$ (this exact formulation seems to
appear first in \cite{Ci}, see also \cite{Sch09,Ozawa13}).  This does not give a
method for deciding positivity, but does show that deciding positivity is in
the class $\Pi_2^0$ for finitely presented archimedean algebras. 

While positivity is decidable for the free group algebras and other
prototypical noncommutative $*$-algebras $\mcA$ listed above, the point of this
paper is to show that positivity becomes undecidable for tensor products $\mcA
\otimes \mcA$ of such algebras:
\begin{theorem}\label{thm:mainbipartite}
    Let $\mcA_\Q$ be one of the algebras $\Q^*\langle x_1,\ldots,x_n\rangle$ or
    $\Q \F_n$ for $n \geq 2$, or $\Q \Z_m^{*n}$ for $n \geq 3$, $m \geq 2$ or
    $n \geq 2$, $m \geq 3$, and let $\mcA = \mcA_\Q \otimes_\Q \C$ be the
    complexification. Then the decision problem
    \begin{equation*}\tag{Pos}
        \text{Given $\alpha \in \mcA_\Q \otimes_{\Q} \mcA_\Q$, is $\alpha$ positive
            as an element of $\mcA \otimes_{\C} \mcA$?}
    \end{equation*}
    is coRE-hard. 
\end{theorem}
Here, $\Q^*\langle x_1,\ldots,x_n\rangle$ denotes the $*$-algebra of noncommutative $*$-polynomials over $\Q$.
Specifically, \Cref{thm:mainbipartite} shows that there is a computable mapping from Turing machines
$M$ to self-adjoint elements $\alpha_M \in \mcA_\Q \otimes_{\Q} \mcA_\Q$, such
that $\alpha_M$ is positive if and only if $M$ does not halt. The bounds on $n$
and $m$ in \Cref{thm:mainbipartite} are tight, as the algebras $\Q^*\ang{x}
\otimes \Q^*\ang{x} = \Q[x_1,x_2]$, $\Q \F_1 \otimes \Q \F_1 = \Q \Z \times
\Z$, and $\Q \Z_m \times \Q \Z_m$ are commutative, and positivity is also
decidable for $\Q \Z_2^{*2} \times \Z_2^{*2}$ (see \Cref{ex:Z22}). As an
immediate corollary of \cref{thm:mainbipartite}, we get:
\begin{cor}\label{cor:mainbipartite}
    Let $\mcA_\Q$ be one of the algebras $\Q^*\langle x_1,\ldots,x_n\rangle$ or
    $\Q \F_n$ for $n \geq 2$, or $\Q \Z_m^{*n}$ for $n \geq 3$, $m \geq 2$ or $n
    \geq 2$, $m \geq 3$, and let $\mcA$ be the complexification. Then there are
    elements of $\mcA_\Q \otimes_{\Q} \mcA_\Q$ which are positive, but are not
    sums of hermitian squares in $\mcA \otimes_{\C} \mcA$, so $\mcA
    \otimes_{\C} \mcA$ does not have the SOS property.
\end{cor}
\begin{proof}
    If every element of $\mcA_\Q \otimes_\Q \mcA_\Q$ is a sum of squares in $\mcA
    \otimes_{\C} \mcA$, then the problem of deciding whether elements of $\mcA
    \otimes \mcA$ are positive is contained in RE by searching through SOS
    decompositions. Since coRE is not contained in RE, this is not possible. 
\end{proof}
For the algebras $\Q \mcF_n$ and $\Q \Z_m^{*n}$, $m \geq 3$, the hypothesis $n
\geq 2$ in \Cref{cor:mainbipartite} is necessary, since finite-dimensional
$C^*$-algebras like $\C \Z_m \times \Z_m$ always have the SOS property, and $\C
\Z \times \Z$ has the SOS property by a result of Scheiderer \cite{Sch06}.
We do not know if $\C \Z_2^{*2} \times \Z_2^{*2}$ has the SOS property. For
$\mcA_{\Q} = \Q^*\ang{x_1,\ldots,x_n}$, the corollary already follows (for all $n
\geq 1$) from the fact that $\C[x,y]$ does not have the SOS property; indeed,
if $p(x,y)$ is a positive polynomial in $\C[x,y]$ which is not a sum of
squares, then $p(x_1\otimes 1, 1 \otimes x_1)$ is positive in $\mcA \otimes
\mcA$, but does not have a sum of squares decomposition. The Motzkin
polynomials provide explicit examples of positive polynomials in $\Q[x,y]$
which are not sums of squares, although of course they are sums of squares of
rational functions by Artin's theorem. \Cref{cor:mainbipartite} does not seem
to have been known for $\mcA_{\Q} = \Q \mcF_n$ and $\Q \Z_m^{*n}$; however,
\Cref{thm:mainbipartite} is stronger, in that SOS decompositions can be
replaced in \Cref{cor:mainbipartite} with any other certificate of positivity.
For instance, for $\mcA_{\Q} = \Q^* \ang{x_1,\ldots,x_n}$ with $n \geq 2$,
there are positive elements of $\mcA_{\Q} \otimes_{\Q} \mcA_{\Q}$ which are not
sums of squares of noncommutative rational functions in the sense of
\cite{KSV23}. The proof of \Cref{thm:mainbipartite} is given in
\Cref{sec:actualmain}. We actually prove a more general statement
(\Cref{thm:actualmain}), which applies to other algebras including the algebra
of bounded contractions.

The proof of \Cref{thm:actualmain} is by reduction to tracial positivity. An
element $\alpha$ of a $*$-algebra $\mcA$ is said to be \emph{trace-positive} if
$\tau(\alpha) \geq 0$ for all (bounded) tracial states $\tau$ on $\mcA$. 
We show that tracial positivity is undecidable for free $*$-polynomials and group
algebras:
\begin{theorem}\label{thm:maintracial}
    Let $\mcA_\Q$ be one of the algebras $\Q^*\langle x_1,\ldots,x_N\rangle$,
    $\Q \F_N$, or $\Q \Z_m^{*N}$ for $m \geq 2$. If $N$ is large enough, then
    the decision problem
    \begin{equation*}\tag{TrPos}
   \text{Given $\alpha \in \mcA_\Q$, is $\alpha$ trace-positive as an element
            of $\mcA$? }  
    \end{equation*}
    is coRE-hard.
\end{theorem}
Klep and Schweighofer \cite{KS08} prove an analog of Helton and McCullough's
Positivstellensatz for the algebra of bounded contractions: an element $\alpha$
is trace-positive if and only if $\alpha+\eps$ is a sum of hermitian squares
and commutators for all $\eps > 0$. This is extended to free group algebras by
Juschenko and Popovych \cite{JP11}, to archimedean algebras by Ozawa
\cite{Ozawa13}, and to $\Q^*\ang{x_1,\ldots,x_n}$ by Klep, Scheiderer, and
Vol\v{c}i\v{c} \cite{KSV23}. Quarez \cite{Qu15} (see also \cite[Example
3.2]{KSV23}) gives examples of trace-positive polynomials in
$\Q^*\ang{x_1,x_2}$ which are not sums of hermitian squares and commutators, but which are sums of squares and commutators of noncommutative
rational functions. As in \Cref{cor:mainbipartite}, \Cref{thm:maintracial}
immediately implies that all the algebras listed in the theorem contain
trace-positive elements which are not sums of squares and commutators of
noncommutative rational functions.  Unlike \Cref{thm:mainbipartite}, we do not
know the smallest bound $N$ such that \Cref{thm:maintracial} holds. 

Part of the motivation for \Cref{thm:mainbipartite} comes from Kirchberg's
reformulation of the Connes embedding problem, sometimes called Kirchberg's
conjecture \cite{Ki}, and the recent MIP*=RE theorem of Ji, Natarajan, Vidick,
Wright, and Yuen \cite{MIPRE}, which resolves this conjecture.  Kirchberg's
conjecture asks whether the max and min tensor norms on $\C \F_n \otimes \F_n =
\C \F_n \times \F_n$ agree for $n \geq 2$ \cite{Ki}. As Ozawa observes
\cite{Ozawa04,Ozawa13}, if $\mcA$ is archimedean and RFD, then the max and min
tensor norms on $\mcA \otimes \mcA$ agree if and only if $\mcA \otimes \mcA$ is
RFD, so Kirchberg's conjecture is equivalent to whether $\C \F_n \otimes \C
\F_n = \C \F_n \times \F_n$ is RFD. Kirchberg's conjecture is still equivalent
to the Connes embedding problem if $\F_n$ is replaced by $\Z_m^{*n}$ for $n
\geq 3$, $m \geq 2$ or $n \geq 2$, $m \geq 3$ \cite{Ozawa13, Fritz}.  
The algebras $\C \Z_m^{*n} \otimes \C \Z_m^{*n}$ are of particular interest in
quantum information, since they correspond to bipartite (meaning composed of
two spatially separated subsystems) measurement scenarios with $n$ measurement
settings and $m$ measurement outcomes per subsystem, and they are closely
connected with the complexity class MIP* of multiprover interactive proof
systems with entangled provers. 

The main result of \cite{MIPRE} is that MIP* is equal to RE, the class of
recursively enumerable languages. In algebraic language, this implies that
there is a computable mapping from Turing machines $M$ to elements $\alpha_M
\in \overline{\Q} \Z_m^{*n} \times \Z_m^{*n}$ (where $\overline{\Q}$ is the algebraic closure of $\Q$, and $m$ and $n$ also depend on $M$), such
that if $M$ halts then $\alpha_M$ is not positive in finite-dimensional
representations of $\C \Z_m^{*n} \times \Z_m^{*n}$, and if $M$ does not halt, then $\alpha_M-1$ is positive in
finite-dimensional representations of $\C \Z_m^{*n} \times \Z_m^{*n}$.  This shows that the promise problem 
\begin{equation*}
    \tag{GapFDPos}
        \begin{minipage}{.7\textwidth}
            Given $n,m \geq 1$ and $\alpha \in \overline{\Q} \Z_m^{*n} \times \Z_m^{*n}$, determine
    if either $\alpha - 1$ is positive in finite-dimensional representations, or
    $\alpha$ is not positive in some finite-dimensional representation, promised that
    one of the two is the case. 
    \end{minipage}
\end{equation*}
is coRE-hard, so it is undecidable to determine if an element of $\C \Z_m^{*n}
\times \Z_m^{*n}$ is strictly positive in finite-dimensional representations.
If $\C \Z_m^{*n} \times \Z_m^{*n}$ were RFD, then (GapFDPos) would be
equivalent to the promise problem
\begin{equation*}
    \tag{GapPos}
        \begin{minipage}{.7\textwidth}
            Given $n,m \geq 1$ and $\alpha \in \overline{\Q} \Z_m^{*n} \times \Z_m^{*n}$, determine
    if either $\alpha - 1$ is positive, or $\alpha$ is not positive,
    promised that one of the two is the case. 
    \end{minipage}
\end{equation*}
Note that $\alpha-1$ is positive if and only if $\alpha$ is strictly positive
with spectrum contained in $[1,+\infty)$, so (GapPos) is a ``gapped'' version
of the decision problem (Pos) considered in \Cref{thm:mainbipartite} (albeit with $\Q$ replaced by $\overline{\Q}$, and $n$ and $m$ allowed to vary). By
Helton and McCullough's Positivstellensatz, if $\alpha - 1$ is positive then $\alpha$ has
an SOS decomposition proving positivity. Due to the promise in (GapPos), a proof that $\alpha$ is positive is also a proof that $\alpha-1$ is positive.  So unlike (Pos), the problem (GapPos)
is contained in RE.  Since RE does not contain coRE, (GapPos) cannot be
coRE-hard, and hence the MIP*=RE theorem implies that $\C \Z_m^{*n} \times
\Z_m^{*n}$ is not RFD for large enough $n$ and $m$.

There is a natural complement to the $\text{MIP}^*=\text{RE}$ theorem called the $\text{MIP}^{\text{co}}=\text{coRE}$
conjecture, which if true would imply (GapPos) is RE-complete, and hence that (Pos) is RE-hard for $\mcA = \C \Z_m^{*n}$ (assuming again that $\Q$ is replaced by $\overline{\Q}$, and $n$ and $m$ are allowed to
vary). If $\mcA\otimes\mcA$ is RFD, then (Pos) is contained in coRE, so showing that (Pos)
is RE-hard for $\mcA = \C \Z_m^{*n}$ (with or without varying $n$ and $m$)
would give another disproof of the Connes embedding problem. Proving that
$\text{(GapPos)}$ is RE-complete would also imply that it is undecidable to determine
if $\alpha \in \C \Z_m^{*n} \times \Z_m^{*n}$ is a sum of squares. On the other
hand, because (GapPos) is contained in RE, hardness results for (GapPos) cannot
be used to prove \Cref{cor:mainbipartite}, and \Cref{thm:mainbipartite} is not
true if (Pos) is replaced by (GapPos). Hence \Cref{thm:mainbipartite} and the
$\text{MIP}^{\text{co}}=\text{coRE}$ conjecture seem to be complementary. In any case, the proof methods we used here are different from the methods used in the proof of $\text{MIP}^*=\text{RE}$, and our techniques do not appear to be applicable to the $\text{MIP}^{\text{co}}=\text{coRE}$ conjecture. 

As mentioned above, the
only upper bound we have on (Pos) is $\Pi^0_2$. If (Pos) contains both RE and
coRE, then it seems possible that it is equal to $\Pi^0_2$, and we conjecture
that this is the case for all the algebras in \Cref{thm:mainbipartite}.
\begin{conjecture}
    Let $\mcA_\Q$ be one of the algebras $\Q^*\langle x_1,\ldots,x_n\rangle$ or
    $\Q \F_n$ for $n \geq 2$, or $\Q \Z_m^{*n}$ for $n \geq 3$, $m \geq 2$ or
    $n \geq 2$, $m \geq 3$, and let $\mcA$ be the complexification. Then the decision problem
    (Pos) from \Cref{thm:mainbipartite} is $\Pi_2^0$-complete.
\end{conjecture}

Tracial positivity has also been studied in connection with the Connes embedding problem \cite{Had,KS08,BDKS,JP11,Ozawa13}, and the above connection with $\text{MIP}^*=\text{RE}$ and $\text{MIP}^{\text{co}}=\text{coRE}$ could be described equally as well in terms of tracial positivity and $\text{MIP}$ protocols with synchronous strategies (see e.g. \cite{HMPS19,KPS18,MNY21,NZ23}). For tracial positivity, we make the following conjecture.

\begin{conjecture}\label{conj:tracepositive}
    Let $\mcA_\Q$ be one of the algebras $\Q^*\langle x_1,\ldots,x_N\rangle$,
    $\Q \F_N$, or $\Q \Z_m^{*N}$ for $m \geq 2$, and let $\mcA$ be the complexification. If $N$ is large enough, then
    the decision problem (TrPos) from \Cref{thm:maintracial} is $\Pi_2^0$-complete.
\end{conjecture}

If \Cref{conj:tracepositive} is true, an additional problem is to find, for each algebra listed above, the smallest bound on $N$ such that the decision problem (TrPos) is $\Pi_2^0$-complete.

\subsection{Acknowledgements}

We thank Igor Klep, Tim Netzer, Andreas Thom, and Jurij Vol\v{c}i\v{c} for
helpful conversations. We thank the anonymous referee for helpful comments, including supplying reference \cite{SS11}. AM is supported by NSERC  Alliance Consortia Quantum grants, reference number: ALLRP 578455 - 22. WS is supported by NSERC DG 2018-03968 and an Alfred P.
Sloan Research Fellowship.

\section{Notation and background}

\subsection{Algebras and states}

We work with both groups and $*$-algebras. If $\K$ is a field with an
involution, then a $*$-algebra over $\K$ is a
unital associative $\K$-algebra $\mcA$ with an antilinear involution $\mcA \arr \mcA : a \mapsto
a^*$ such that $(a b)^* = b^* a^*$ for all $a,b \in \mcA$.  A $*$-homomorphism
$\pi : \mcA \to \mcB$ between $*$-algebras $\mcA$ and $\mcB$ is an algebra homomorphism 
such that $\pi(a^*) = \pi(a)^*$. 
In this paper, $\K$ will be either $\C$ or
$\Q$ (with trivial involution), with $\C$ being the default choice.
When $\mcA$ is a $*$-algebra over $\Q$, we use $\mcA_\C$ to denote $\mcA\otimes_\Q\C$. We use $\Id$ for the identity in both groups
and $*$-algebras. 

If $G$ is a group, then we let $\K G$ be the group algebra of $G$ over the field $\K$. Recall that
this is a $*$-algebra with $g^* = g^{-1}$. If $S$ is a set, then $\F_S$ will
denote the free group with generating set $S$, and $\F_n :=
\F_{\{x_1,\ldots,x_n\}}$. We use the standard notation for group presentations,
in which if $S$ is a set, and $R \subseteq \F_S$, then $\langle S: R \rangle$
denotes the quotient of $\F_S$ by the normal subgroup generated by $R$.  We
often want to add generators and relations to a given presentation $G = \langle
S : R \rangle$, so if $S_1$ is another set disjoint from $S$, and $R_1
\subseteq \F_{S \cup S_1}$, then we define $\langle G, S_1 : R_1 \rangle :=
\langle S \cup S_1 : R \cup R_1 \rangle$. If $G$ is a group, then $G^{*n}$
denotes the free product of $n$ copies of $G$. The cyclic group of order $k$
will be denoted by $\Z_k$.  If $S$ is a set, then $\Z_k^{*S}$ denotes the free
product of $|S|$ copies of $\Z_k$, where the generators of the factors are
denoted by the elements of $S$. In other words, $\Z_k^{*S} = \langle S : s^k =
1 \text{ for all } s \in S\rangle$.

We also define $*$-algebras by means of presentations. If $S$ is a set, we let
$\K^*\ang{S}$ denote the free noncommutative $*$-algebra with generating set
$S$ over $\K$. Elements of $\K^*\ang{S}$ are noncommutative $*$-polynomials over $\K$. For instance, if $S=\lbrace x,y \rbrace$, then elements of $\K^*\ang{S}$ are noncommutative polynomials in $x,x^*,y,y^*$, and the involution $a \mapsto a^*$ exchanges $x$ with $x^*$ and $y$ with $y^*$. If $R \subseteq \K^*\ang{S}$, then $\K^*\ang{S : R}$ denotes the
quotient of $\K^*\ang{S}$ by the two-sided $*$-ideal generated by $R$. As with
groups, if $\mcA = \K^*\ang{S : R}$ is a $*$-algebra presentation,
$S_1$ is a set disjoint from $S$, and $R_1 \subseteq \K^*\ang{S \cup
S_1}$, then 
\begin{equation*}
    \K^*\ang{ \mcA, S : R_1 } := \K^*\ang{ S \cup S_1 : R \cup R_1 }. 
\end{equation*}
If $G = \langle S : R \rangle$ is a group presentation, then 
\begin{equation*}
    \K G = \K^*\ang{ S : 1 - r, r \in R, 1 - s s^*, 1 - s^* s, s \in S}. 
\end{equation*}
For both group and $*$-algebra presentations, we'll follow the standard
practice of writing $a=b$ for the relation $b^{-1} a$ (in the case of groups)
or $a-b$ (in the case of $*$-algebras). Also, a presented group (resp. algebra)
is a group (resp. algebra) with a fixed presentation, and a finitely presented
group (resp. algebra) is a presented group (resp. algebra) where the
presentation has a finite number of generators and relations.

As mentioned above, unless specified otherwise a $*$-algebra will mean a
$*$-algebra over $\C$.  If $\mcH$ is a Hilbert space, then $\mcB(\mcH)$ is the
$C^*$-algebra of bounded operators on $\mcH$. A representation of a $*$-algebra
$\mcA$ is a $*$-homomorphism $\mcA \to \mcB(\mcH)$ for some Hilbert space
$\mcH$. A representation is said to be finite-dimensional if $\mcH$ is
finite-dimensional (we mention this notion in the preliminary but we won't work
with it in this paper). As mentioned in the introduction, a self-adjoint
element $a \in \mcA$ is (representation) positive if and only if $\phi(a)$ is
positive in $\mcB(\mcH)$ for all Hilbert spaces $\mcH$ and representations
$\phi : \mcA \to \mcB(\mcH)$. 

A (bounded) state on a $*$-algebra $\mcA$ is a linear function $\psi : \mcA
\arr \C$ with  $\psi(1)=1$ such that  $\psi(a^*) = \overline{\psi(a)}$,
$\psi(a^* a) \geq 0$, and 
\begin{align}\label{eq:bounded}
    \sup\left\{\frac{\psi(b^* a^* a b)}{\psi(b^* b)}: b \in \mcA, \psi(b^* b) \neq 0\right\}<\infty
\end{align}
for all $a\in\mcA$.
The GNS theorem for $*$-algebras states that $\psi$ is a state if and only if
there is a representation $\pi : \mcA \arr \mcB(\mcH)$ and a unit vector $v \in
\mcH$ such that $\psi(a) = \langle v,\pi(a)v \rangle$ for all $a \in \mcA$
\cite[Theorem 4.38]{Sch20}. Given a state $\psi$, there is always such a representation $\pi$
and vector $v$ such that $\pi(\mcA) v$ is dense in $\mcH$. A pair $(\pi,v)$
such that $\pi(\mcA) v$ is dense in $\mcH$ and $\langle v,\pi(a) v \rangle =
\psi(a)$ for all $a \in \mcA$ is called a GNS representation of $\psi$. The
GNS representation theorem implies that a self-adjoint element $a \in \mcA$ is 
positive if and only if $f(a) \geq 0$ for all states $f$ on $\mcA$.

A state $\psi$ on a $*$-algebra $\mcA$ induces a semi-norm
$\norm{a}_\psi:=\sqrt{\psi(a^*a)}$ on $\mcA$ satisfying $\abs{\psi(ab)}\leq
\norm{a}_\psi\norm{b}_{\psi}$ for all $a,b\in \mcA$. A state $\psi$ on $\mcA$
is tracial if $\psi(ab)=\psi(ba)$ for all $a,b\in\mcA$. If $\psi$ is tracial,
then $\norm{\cdotp}_{\psi}$ is unitarily invariant, meaning that
$\norm{uav}_{\psi} = \norm{a}_{\psi}$ for all unitaries $u,v \in \mcA$ and
elements $a \in \mcA$.  For an arbitrary state $\psi$, the seminorm
$\norm{\cdotp}_{\psi}$ is left unitarily invariant, but not necessarily right
unitarily invariant.

An element $a$ of a $*$-algebra $\mcA$ is a hermitian square if there is $b \in
\mcA$ such that $a = b^* b$, and a sum of (hermitian) squares if there is
$b_1,\ldots,b_k \in \mcA$ such that $a = \sum_i b_i^* b_i$. If $a$ is a sum of
squares, then $\psi(a) \geq 0$ for all states $\psi$.  If $a_1$ and $a_2$ are
sums of squares, then $\psi(a_1 a_2) \geq 0$ for all tracial states $\psi$,
since $\psi(b_1^* b_1 b_2^* b_2) = \psi((b_2 b_1^*)^* (b_2 b_1^*))$ when $\psi$
is tracial. 

Every $*$-algebra $\mcA$ has an operator norm
$\norm{\cdotp}_\mcA:\mcA\arr\R_{\geq 0}\cup \{\infty\}$ defined by
$\norm{a}_\mcA:=\sup\{\sqrt{\psi(a^*a)}: \psi \text{ a state on } \mcA\}$.
Using the GNS theorem, it can be shown that
$\norm{a}_\mcA=\sup\{\norm{\pi(a)}_{\mcB(\mcH)}:\pi \text{ a }
*\text{-representation} \text{ on a Hilbert space }\mcH\}$ and
$\psi(b^*a^*ab)\leq \norm{a}^2_{\mcA}\psi(b^*b)$ for all $a,b\in\mcA$ and
states $\psi$ on $\mcA$.

A $*$-algebra $\mcA$ is said to be archimedean if for all $a \in \mcA$ (or
equivalently, all $a$ in a generating set for $\mcA$), there is a scalar
$\lambda > 0$ such that $\lambda - a^* a$ is a sum of squares. If $\mcA$ is
archimedean, then a linear functional $f : \mcA \to \C$ satisfying $f(1) = 1$
and $f(a^*) = \overline{f(a)}$, $f(a^* a) \geq 0$ for all $a \in \mcA$
automatically satisfies the boundedness condition in Equation
\eqref{eq:bounded}.  Also, the operator norm $\norm{\cdotp}_{\mcA}$ takes only
finite values, and (as mentioned in the introduction) Helton and McCullough's
Positivstellensatz states that $a \in \mcA$ is positive if and only if $a +
\eps$ is a sum of squares for all $\eps > 0$.  Group algebras $\C G$ are
archimedean, but free $*$-algebras $\C^* \ang{S}$ are not when $S$ is
non-empty, so not all the algebras we deal with are archimedean. More
background can be found in \cite{Sch09} and \cite{Ozawa13}.

\subsection{Computability}
If $\Sigma$ is a set, a word (also called a string) of length $n$ over $\Sigma$ is a finite sequence
$a_1 a_2 \cdots a_n$, where $a_i \in \Sigma$ for all $1 \leq i \leq n$. We use
the standard terminology from the theory of computation (see, e.g. \cite{TCS}). In
particular, a subset $\mcL$ of words over a finite set $\Sigma$ is in RE if there is a
Turing machine $M$ with input alphabet $\Sigma$ which recognizes $\mcL$,
meaning that there is a computation path where $M$ halts on input $x$ if and
only if $x \in \mcL$. In this definition $M$ can be non-deterministic, but
every language in RE can be recognized by a deterministic Turing machine.
A language $\mcL$ is in coRE if its complement $\overline{\mcL}$ is in RE, and
is decidable if it is in both RE and coRE.  Equivalently, $\mcL$ is decidable
if there is a deterministic Turing machine that halts on every input $x$,
returning ``accept'' if $x \in \mcL$ and ``reject'' otherwise. Since many types
of objects can be encoded as strings, we often talk about sets of other objects
being RE or coRE. For instance, the set
\begin{equation*}
    \HALT = \{ (M,x) \text{ a Turing machine}, x \text{ an input to } M \ |\ M
                \text{ halts on input } x \}
\end{equation*}
is the prototypical example of an undecidable RE set, although formalizing it
as a language requires choosing some encoding of pairs $(M,x)$ as strings. The
complement $\coHALT$ of pairs $(M,x)$ such that $M$ does not halt on input $x$
is the prototypical example of an undecidable coRE set. A function $f$ from
strings on $\Sigma$ to strings on $\Sigma'$ is computable if there is a
deterministic Turing machine with input alphabet $\Sigma$ and
output alphabet $\Sigma'$ which takes input $x$ to $f(x)$. A language $\mcL$ of
strings over $\Sigma$ is RE-hard (resp. coRE-hard) if there is a computable
function $f$ from pairs of Turing machines $(M,x)$ to strings over $\Sigma$,
such that $f(M,x) \in \mcL$ if and only if $M$ halts on input $x$ (sometimes
this is called RE-hard under many-one reductions).  The definition of RE and coRE does not depend on the choice of alphabet $\Sigma$, because strings over a finite set can be reencoded as strings over $\{0,1\}$. Strings over $\{0,1\}$ are called bit strings or binary strings.  Natural numbers can be encoded as binary strings as well, so we can also talk about RE, coRE, RE-hard, and coRE-hard subsets of $\mbN$.

We also care about how many steps it takes a Turing machine to accept an input.
The time function of a non-deterministic Turing machine $M$ is the function $T
: \mbN \arr \mbN $, where $T(n)$ is the smallest possible integer such that,
given an input string $x$ of length $\leq n$, if $M$ halts on $x$, then some
computation path of $M$ halts in $\leq T(n)$ steps. We use the convention that
every Turing machine takes at least one step to halt, so $T(n) \geq 1$ for all
$n$.  We say that $\mcL$ can be recognized in time $T' : \mbN \arr \mbN$ if
there is a Turing machine $M$ recognizing $\mcL$ with time function $T(n) \leq
T'(n)$ for all $n$. In particular, $\mcL$ can be recognized in polynomial time
if it can be recognized in time $C n^k$ for some $C, k \in \mbN$. Similarly, we
say that $\mcL$ can be recognized in deterministic time $T' : \mbN \arr \mbN$
if there is a deterministic Turing machine $M$ recognizing $\mcL$ with time
function $T(n) \leq T'(n)$ for all $n$. A function $f$ is computable in time
$T(n)$ if there is a deterministic Turing machine computing $f$ which halts in
$\leq T(n)$ steps on all input strings of length $\leq n$.

The classes RE and coRE are the first level of a larger hierarchy of classes called the arithmetical hierarchy. Although we don't use it in the rest of the paper, the class $\Pi_2^0$ mentioned in the introduction is in the second level of this hierarchy. 
A language $\mcL$ is in $\Pi_2^0$ if there is a Turing machine $M$ taking a pair of strings $x$ and $y$ as input, such that
$x\in\mcL$ if and only if $M$ halts on input $(x,y)$ for all $y$. As discussed in the introduction, the Helton-McCullough Positivstellensatz states that an element $\alpha$ of an archimedean $*$-algebra $\mcA$ is positive if and only if $\alpha+\eps$ is a sum of squares for all $\eps>0$. When $\mcA$ is finitely presented, one can effectively search through all sums of squares. So deciding positivity for finitely presented archimedean algebras is in $\Pi_2^0$.
\subsection{Group theory}\label{ss:grouptheory}

In addition to the elementary definitions above, we use some more advanced
concepts from group theory.  Suppose $S$ is a finite set of generators for a
group $G$, and let $\mcL$ be the set of words over $S \cup S^{-1}$ which are
trivial in $G$.  The word problem for $G$ is the problem of determining whether
a word over $S \cup S^{-1}$ belongs to $\mcL$. Accordingly, we say that the
word problem for $G$ with respect to $S$ is solvable in non-deterministic time
$T : \mbN \arr \mbN$ (resp. non-deterministic polynomial time) if $\mcL$ can be recognized in time $T$
(resp. polynomial time). 

If $G$ has generating set $S$ (not necessarily finite), a normal form for $G$
with respect to $S$ is a function $\eta$ from words over $S \cup S^{-1}$ to
words over $S \cup S^{-1}$, such that $\eta(x) = x$ in $G$ for all words $x$,
and two words $x$ and $y$ are equal in $G$ if and only if $\eta(x) = \eta(y)$.
Equivalently, a normal form can be specified as a function $\eta$ from $G$
to words over $S \cup S^{-1}$, such that $\eta(g) = g$ in $G$ for all $g \in G$.

We denote the free product of a family of
groups $\{G_i\}_{i \in I}$ by $\Asterisk_{i \in I} G_i$, and the 
product by $\prod_{i \in I} G_i$. If $\{G_i\}_{i \in I}$ is a family of
groups, $K$ is a group, and $\phi_i : K \arr G_i$ are injective homomorphisms, 
then the amalgamated free product of the family over $K$ is 
\begin{equation*}
    {}_K \Asterisk_i G_i := \langle \Asterisk_{i \in I} G_i : \phi_i(k) = \phi_j(k) \text{ for all } k \in K, i,j \in I \rangle.
\end{equation*}
If $\phi_i(K)$ is in the centre of $G_i$ for
all $i \in I$, then the central product of this family over $K$ is defined to
be
\begin{equation*}
    \sideset{_K}{_i}\prod G_i := \langle \prod_i G_i : \phi_i(k) = \phi_j(k) \text{ for all } k \in K, i,j \in I \rangle.
\end{equation*}
In both cases, if we identify $K$ with $\phi_i(K)$, then the image of $K$
inside ${}_{K} \Asterisk_{i} G_i$ and $\sideset{_K}{}\prod_{i} G_i$ is independent
of what $i$ we pick. Note that given a generating set $S$ for $K$, we only need the relations $\phi_i(k)=\phi_j(k)$ for all $k\in S$ in the above presentations.

For each $i \in I$, let $A_i \subset G_i$ be a set of right coset
representatives for $\phi_i(K)$, where we pick the identity element $\Id$ as
the representative of $\phi_i(K)$ itself.  The normal form theorem for
amalgamated free products states that every element $g \in {}_{K} \Asterisk_{i} G_i$
can be written uniquely as $g = k g_1 \cdots g_m$, where $k \in K$, and $g_j
\in A_{i_j} \setminus \{1\}$ for some sequence $i_1,\ldots,i_m \in I$ with $i_{j} \neq i_{j+1}$
for all $1 \leq j < m$. While the normal form theorem for amalgamated free products does not give us a normal form of $g$ in the above sense, if $\eta$ is a normal form for $K$ with respect to generating
set $S$, and $\eta_i$ is a normal form for $G_i$ with respect to generating set
$S_i$, $i \in I$, then $g \mapsto \eta(k) \eta_{i_1}(g_1)
\cdots \eta_{i_m}(g_m)$ is a normal form for ${}_{K} \Asterisk_{i} G_i$ with respect to
generating set $S \sqcup \coprod_{i \in I} S_i$ (where 
$\coprod$ is the disjoint union). 
Similarly, if we pick an ordering for $I$, the normal form theorem for central products
states that every element $g \in \sideset{_K}{}\prod_{i} G_i$ can be written
uniquely as $k g_1 \cdots g_m$, where $k \in K$, and $g_i \in A_{i_j} \setminus \{1\}$ for some
increasing sequence $i_1 < i_2 < \ldots < i_m$ in $I$. Again, if $\eta$ is a normal
form for $K$ and $\eta_i$ is a normal form for $G_i$, $i \in I$, then $g \mapsto \eta(k) \eta_{i_1}(g_1) \cdots \eta_{i_m}(g_m)$ is a normal form for
$\sideset{_K}{}\prod_{i} G_i$. In both cases, the normal form theorem implies
that the natural inclusions $G_j \incl {}_K \Asterisk_i G_i$ and $G_j \incl
\sideset{_K}{}\prod_{i} G_i$ are injections for all $j$, so $G_j$ can be
regarded as a subgroup of the amalgamated free product or central product. 
One application of amalgamated free products that we will use is the following
lemma:
\begin{lemma}\label{lem:involution}
    Let $G = \ang{S:R}$. For any $y \in S$, $G$ is a subgroup of 
    \begin{equation*}
        H := \ang{G, s, t : y = st, s^2=t^2=1}
    \end{equation*}
    via the natural inclusion.
\end{lemma}
\begin{proof}
    Suppose $y$ has order $k \in \mbN \cup \{+\infty\}$. Then $H$ is the 
    amalgamated free product of $G$ with the dihedral group
    \begin{equation*}
        D_k = \ang{s,t : s^2=t^2=(st)^k = 1 }
    \end{equation*}
    of order $2k$ over the subgroups $\ang{y} \iso \Z_k \iso \ang{s,t}$.
\end{proof}

Another construction we'll use are HNN extensions. Suppose that $G$ is a group,
$H$ and $K$ are subgroups of $G$, and $\phi : H \to K$ is an
isomorphism. Then the HNN extension of $G$ by $\phi$ is
\begin{equation*}
    \langle G, t : t h t^{-1} = \phi(h) \text{ for all } h \in H \rangle.
\end{equation*}
It is well-known that the order of the generator $t$ is infinite, and that,
similar to the amalgamated free product and central product, $G$ is a
subgroup of the HNN extension via the natural inclusion. More generally,
given a sequence of subgroups $H_1,\ldots,H_m$ and $K_1,\ldots,K_m$ of a group $G$,
along with isomorphisms $\phi_i : H_i \arr K_i$, $1 \leq i \leq m$, $G$ will be
subgroup of the iterated HNN extension
\begin{equation*}
    \langle G, t_1,\ldots,t_m : t_i h t_i^{-1} = \phi_i(h) \text{ for all } 1 \leq i \leq m
            \text{ and } h \in H_i \rangle.
\end{equation*}
Britton's lemma allows us to determine the trivial words in an iterated HNN extension.
 \begin{lemma}[Britton's lemma, Theorem 11.81 of \cite{Rotman}]\label{lemma:britton}
    Let $H_1,\ldots,H_m$ and $K_1,\ldots,K_m$ be subgroups of a group $G$, and
    suppose $\phi_i : H_i \arr K_i$ is an isomorphism for $i=1\ldots,m$. Let
    \begin{equation*}
        L = \langle G, t_1,\ldots,t_m : t_i h t_i^{-1} = \phi_i(h) \text{ for all } 1 \leq i \leq m
                \text{ and } h \in H_i \rangle
    \end{equation*}
    be the iterated HNN extension, and suppose 
    \begin{equation*}
        w = g_0 t_{i_1}^{e_1} g_1 t^{e_2}_{i_2} \cdots t^{e_n}_{i_n} g_n \in L
    \end{equation*}
    where $g_0,\ldots,g_n \in G$, $e_1,\ldots,e_n \in \{\pm 1\}$, and $1 \leq i_1,\ldots,i_n \leq m$.
    If $w = \Id$ in $L$, then at least one of the following holds:
    \begin{enumerate}[(a)]
        \item $n=0$ and $g_0 = \Id$ in $G$, 
        \item there is $1 \leq j < n$ such that $e_j = 1 = - e_{j+1}$, $i_j = i_{j+1}$, and $g_i \in H_{i_j}$, or
        \item there is $1 \leq j < n$ such that $e_j = -1 = - e_{j+1}$, $i_j = i_{j+1}$, and $g_i \in K_{i_j}$.
    \end{enumerate}
\end{lemma}

Let $G = \langle S : R \rangle$ be a finitely presented group. A function $f :
\mbN \arr \mbN$ is an isoperimetric function for $G$ if for all words 
$w = s_1 \cdots s_n$ of length $n \geq 1$ over $S \cup S^{-1}$,
if $w = \Id$ in $G$ then there are elements $u_1,\ldots,u_k \in \F_S$ and
$r_1,\ldots,r_k \in R \cup R^{-1}$ with $k \leq f(n)$ and
\begin{equation*}
    w = u_1 r_1 u_1^{-1} \cdot u_2 r_2 u_2^{-1} \cdots u_k r_k u_k^{-1}.
\end{equation*}

\begin{lemma}[\cite{Ge93}]\label{lemma:isod}
Let $G=\ang{S:R}$ be a finitely presented group with an isoperimetric function $f$, and let $l:=\max\{\abs{r}:r\in R\}$ be the length of the longest relation in $R$. If a word $w\in\F_S$ of length $n$ is equal to $1$ in $G$, then there are $z_1,\cdots,z_k\in\F_S$ and $r_1,\cdots,r_k\in R\cup R^{-1}$ with $k\leq f(n)$ and $\abs{z_i}\leq kl+l+n$ for all $1\leq i\leq k$ such that
\begin{align*}
    w = z_1 r_1 z_1^{-1}z_2 r_2 z_2^{-1} \cdots z_k r_k z_k^{-1}.
\end{align*}
in $\F_{S}$.
\end{lemma}

We use the following theorem of Birget, Ol'shanskii, Rips, and Sapir:
\begin{theorem}[\cite{BORS02}]\label{thm:BORS}
    Let $G$ be a group with finite generating set $S$. If the word problem for
    $G$ with respect to $S$ is solvable in non-deterministic time $T(n)$, where $T^4$ is
    superadditive, then $G$ can be embedded into a finitely presented group
    $H$ with isoperimetric function $n^2 T(n^2)^4$. 
\end{theorem}
Here a function $T : \mbN \arr \mbN$ is superadditive if $T(m+n) \geq T(m) +
T(n)$ for all $m,n$.  It follows immediately from this theorem that if the word
problem for $G$ can be solved in non-deterministic polynomial time, then $G$ can be embedded in a finitely presented
group $H$ with a polynomial isoperimetric function. 

\section{Approximate representations}\label{sec:approxrepn}

We introduce some terminology for approximate states.

\begin{definition}\label{defn:erstate}
    Let $\mcA$ be a finitely generated $*$-algebra. Suppose $\eps \geq 0$ and 
    $\mcR\subseteq \mcA$. An \emph{$(\epsilon,\mcR)$-state on $\mcA$} is a state
    $\varphi$ on $\mcA$ such that $\varphi(r^*r)\leq \epsilon$ for all
    $r\in\mcR \cup \mcR^*$.  If, in addition, $\varphi$ is tracial, then we say that
    $\varphi$ is a \emph{tracial $(\epsilon,\mcR)$-state on $\mcA$}.
\end{definition}

If $(\pi : \mcA \to
\mcB(\mcH), \xi)$ is the GNS representation of a state $\varphi$ on $\mcA$, then
$\norm{x}_{\xi}:=\norm{x \cdot \xi}$ is a semi-norm on $\mcB(\mcH)$ such that
$\norm{a}_\varphi=\norm{\pi(a)}_\xi$ for all $a\in\mcA$. If $\varphi$ is an
$(\epsilon,\mcR)$-state on $\mcA$ as in \cref{defn:erstate}, then
$\norm{\pi(r)}_\xi\leq \sqrt{\epsilon}$ for all $r\in\mcR$, so we can think of
$\pi$ as a $\sqrt{\epsilon}$-approximate representation for $\C^*\ang{\mcA :
\mcR}$ with respect to $\norm{\cdot}_{\xi}$. 

\begin{definition}\label{defn:rde}
    Let $\mcA$ be a $*$-algebra with generating set
    $\mcX$. Let $\mcR\subseteq\C^*\ang{\mcX}$ be a set of $*$-polynomials over
    $\mcX$. For any $*$-polynomial $f\in\C^*\ang{\mcX}$ that is trivial in
    $\mcA/\ang{\mcR}$, we say that $\sum_{i=1}^n \lambda_iu_ir_iv_i$ is an
    $\mcR$-decomposition for $f$ in $\mcA$ if
    \begin{enumerate}
        \item $u_i,v_i$ are $*$-monomials in $\C^*\ang{\mcX}$ for all $1\leq i\leq n$, 
        \item $r_i\in\mcR\cup\mcR^*$ for all $1\leq i\leq n$,
        \item $\lambda_i \in \C$ for all $1\leq i\leq n$, and
        \item $f=\sum_{i=1}^n \lambda_iu_ir_iv_i$ in $\mcA$.
    \end{enumerate}
    The size of an $\mcR$-decomposition $\sum_{i=1}^n \lambda_i u_i r_i v_i$ is $\sum_{i=1}^{n} |\lambda_i| (1+ \norm{r_i}_{\mcA} \deg(v_i))$,
    where $\norm{\cdot}_{\mcA}$ is the operator norm in $\mcA$.
 \end{definition}
 Here $\C^*\ang{\mcX}$ means the set of formal $*$-polynomials with indeterminates $\mcX$. There is a natural map from $\C^*\ang{\mcX}\arr\mcA$ sending $x\in \mcX$ to its image in $\mcA$,  so we can also regard elements of $\C^*\ang{\mcX}$ as elements of $\mcA$, e.g. when taking the quotient $\mcA/ \ang{\mcR}$. A $*$-monomial in $\C^*\ang{\mcX}$ is a product $a_1a_2\cdots a_k$ where
$k \geq 0$ and $a_1,\ldots, a_k\in\mcX\cup\mcX^*$ (the integer $k$ is called
the degree). 

We let $\norm{\cdot}_1$ and $\norm{\cdot}_{1,1}$ denote the $1$-norm
and first Sobolev $1$-seminorm on $\C^*\ang{\mcX}$ with respect to the basis of
$*$-monomials.  In other words, if $f = \sum_{i=1}^n \lambda_i u_i \in
\C^*\ang{\mcX}$ for distinct $*$-monomials $u_1,\ldots,u_n$, then $\norm{f}_1
:= \sum_{i=1}^n |\lambda_i|$ and $\norm{f}_{1,1} := \sum_{i=1}^n |\lambda_i|
\deg(u_i)$. Note that $\norm{\cdot}_1$ is submultiplicative, and that
$\norm{f^*}_1 = \norm{f}_1$, $\norm{f^*}_{1,1} = \norm{f}_{1,1}$ for all $f 
\in \C^*\ang{\mcX}$.
The following elementary lemma is useful for computing sizes of
$\mcR$-decompositions.
\begin{lemma}\label{lemma:sizecalc}
    Let $\mcA$ be a $*$-algebra with generating set $\mcX$, and let $\mcR
    \subseteq \C^*\ang{\mcX}$ be a set of $*$-polynomials over $\mcX$.
    Then:
    \begin{enumerate}[(a)]
        \item If $f = \sum_{i=1}^n u_i r_i v_i$ for $*$-polynomials $u_1,\ldots,u_n,v_1\ldots,v_n \in \C^*\ang{X}$
            and $r_1,\ldots,r_n \in \mcR \cup \mcR^{*}$, then $f$ has an $\mcR$-decomposition of size at most
            \begin{equation*}
                \sum_{i=1}^n \norm{u_i}_1 \norm{v_i}_1 + \norm{u_i}_1 \norm{r_i}_\mcA \norm{v_i}_{1,1}.
            \end{equation*}
        \item If $f-g$ has an $\mcR$-decomposition of size $\leq \Lambda_1$, and $g-h$ has an $\mcR$-decomposition
            of size $\leq \Lambda_2$, then $f-h$ has an $\mcR$-decomposition of size $\leq \Lambda_1 + \Lambda_2$.
        \item If $g$ has an $\mcR$-decomposition of size $\leq \Lambda$, then $f g$ has an $\mcR$-decomposition
            of size $\leq \norm{f}_1 \Lambda$ for all $f \in \C^* \ang{X}$. 
    \end{enumerate}
\end{lemma}
For instance, suppose we can demonstrate that $g = h$ in $\mcA / \ang{\mcR}$ by
finding $f_1,\ldots,f_n$, $u_1,\ldots,u_n$, $v_1,\ldots,v_n \in
\C^*\ang{\mcX}$ and $r_1,\ldots,r_n \in \mcR \cup \mcR^*$ such that $f_{i+1} -
f_i = u_i r_i v_i$ for all $0 \leq i \leq n$, where $f_0 = g$ and $f_{n+1} =
h$. Then we know that $g-h$ has an $\mcR$-decomposition of size at most
$\sum_{i=1}^n \norm{u_i}_1 \norm{v_i}_1 + \norm{u_i}_1 \norm{r_i}_{\mcA}
\norm{v_i}_{1,1}$. 

If $\mcX$ is a set of unitary generators for $\mcA$, then all
$*$-monomials in $\C^*\ang{\mcX}$ are unitary in $\mcA$.  As a result, an
$\mcR$-decomposition of $f$ can be used to bound the state-dependent norm of
$f$ with respect to any tracial $(\eps,\mcR)$-state: 
\begin{proposition}\label{prop:tracerbound}
    Suppose  $\mcA$ is a $*$-algebra generated by a finite set of unitaries $\mcX$.
    Let $\mcR\subset\C^*\ang{\mcX}$ be a set of $*$-polynomials over $\mcX$, and
       let $\tau$ be a tracial $(\epsilon,\mcR)$-state on $\mcA$. If
    $f\in\C^*\ang{\mcX}$ is trivial in $\mcA/\ang{\mcR}$ with an
    $\mcR$-decomposition $\sum_{i=1}^n \lambda_iu_ir_iv_i$, then
    $\norm{f}_{\tau}\leq (\sum_{i=1}^n|\lambda_i|)\sqrt{\epsilon}$.
\end{proposition}
Note that $\sum_{i=1}^n |\lambda_i|$ is less than or equal to the size of $\sum_{i=1}^n \lambda_i u_i r_i v_i$.
\begin{proof}
    Since $\norm{\cdotp}_{\tau}$ is unitarily invariant, 
    \begin{align*}
       \norm{f}_{\tau}\leq
    \sum_{i=1}^n|\lambda_i| \norm{u_ir_iv_i}_{\tau} =\sum_{i=1}^n|\lambda_i|
    \norm{r_i}_{\tau}   \leq (\sum_{i=1}^n|\lambda_i|)\sqrt{\epsilon}. 
    \end{align*}
\end{proof}
The proof of \Cref{prop:tracerbound} does not work if $\tau$ is not tracial.
However, we can still bound $\norm{f}_{\tau}$
using the size of an $\mcR$-decomposition if $\tau$ is the restriction of a state
on $\mcA \otimes \mcA$ satisfying a synchronicity\footnote{The term synchronous is borrowed from nonlocal games in quantum information. Roughly speaking, a state $\varphi$ is synchronous if the action of $x\in\mcX$ on $\varphi$ is the same in both tensor factors.} condition:
\begin{definition}\label{defn:syncstate}
    Suppose $\mcA$ is a $*$-algebra.
    A state $\varphi$ on $\mcA \otimes \mcA$ is $(\eps,\mcX)$-synchronous for
    some $\eps \geq 0$ and $\mcX \subseteq \mcA$ if
    \begin{equation}
        \norm{x \otimes 1 - 1 \otimes x}_{\varphi} \leq \sqrt{\eps}\label{eq:x11x}
    \end{equation}
    for all $x \in \mcX$.
\end{definition}

 When working with $(\eps,\mcX)$-synchronous states, it is convenient to let $\omega_\mcX$ denote the linear map  $\C^*\ang{\mcX}\arr\C^*\ang{\mcX}$ sending $x_1x_2\cdots x_n\mapsto x_nx_{n-1}\cdots x_1$ for all  $n\in\N$ and $x_1,\ldots, x_n\in\mcX\cup\mcX^*$. This map satisfies $\omega_\mcX(a^*)=\omega_\mcX(a)^*$ and $\omega_\mcX(ab)=\omega_\mcX(b)\omega_\mcX(a)$ for all $a,b\in\C^*\ang{\mcX}$, and hence $\omega_\mcX$ is a $*$-isomomorphism $\C^*\ang{\mcX}\arr\C^*\ang{\mcX}^{\mathrm{op}}$, where $\C^*\ang{\mcX}^{\mathrm{op}}$ is the opposite algebra of $\C^*\ang{\mcX}$.

\begin{lemma}\label{lemma:xy_stability}
    Let $\mcA$ be a $*$-algebra generated by a finite set of unitaries $\mcX$. If $\varphi$ is an $(\eps,\mcX)$-synchronous state on $\mcA \otimes\mcA$, then 
    \begin{align}
        \norm{u\otimes \Id-\Id\otimes \omega_\mcX(u)}_\varphi\leq \norm{u}_{1,1}\sqrt{\eps}\label{eq:xy_stability}
    \end{align}
    for all $u\in\C^*\ang{\mcX}$.
\end{lemma}

\begin{proof}
For every $x\in \mcX$, since $x$ is unitary and $\norm{\cdotp}_\varphi$ is left
unitarily invariant, 
\begin{equation*}
    \norm{x^*\otimes \Id -\Id \otimes x^*}_\varphi=\norm{(x\otimes x)(x^*\otimes \Id -\Id \otimes
x^*)}_\varphi=\norm{x\otimes \Id-\Id\otimes x}_\varphi\leq \sqrt{\eps}.
\end{equation*}
 Thus \Cref{eq:x11x} holds for all $x\in\mcX\cup \mcX^*$. Suppose $u=x_1\cdots x_n$ is a monomial in $\C^*\ang{\mcX}$ of degree $n\geq 1$. Let
$x_0 = x_{n+1}=\Id$. Then  
\begin{align*}
    \norm{u\otimes \Id-\Id\otimes \omega_\mcX(u)}_\varphi & = 
\norm{\sum_{k=1}^n (x_0 x_1\cdots x_{k-1}\otimes x_{n+1} x_n \cdots x_{k+1})(x_k\otimes \Id-\Id\otimes x_k)  }_\varphi \\
    & \leq \sum_{k=1}^n\norm{x_k\otimes \Id-\Id\otimes x_k}_\varphi \leq n\sqrt{\eps}.
\end{align*}
By the subadditivity of $\norm{\cdotp}_\varphi$, \Cref{eq:xy_stability} holds for all $u\in\C^*\ang{\mcX}$.
\end{proof}

The following lemmas show that some of the properties of tracial states hold approximately for $(\eps,\mcX)$-synchronous states.

\begin{lemma}\label{lemma:epstracial}
    Let $\mcA$ be a $*$-algebra generated by a finite set of unitaries $\mcX$, and let $\varphi$ be an $(\eps,\mcX)$-synchronous state on $\mcA \otimes
    \mcA$. If $u$ is a monomial in $\C^*\ang{\mcX}$ and $a \in \mcA$, then $|\norm{u^* a u \otimes 1}_{\varphi} - \norm{a \otimes 1}_{\varphi}| \leq \norm{a}_{\mcA} \deg(u) \sqrt{\eps}$.
\end{lemma}
\begin{proof} 
Since $u$ is a monomial in $\C^*\ang{\mcX}$, $u$ and $\omega_\mcX(u)$ are unitary in $\mcA$. By \Cref{lemma:xy_stability}, $\norm{u\otimes \Id-\Id\otimes \omega_\mcX(u)}_\varphi\leq \deg(u)\sqrt{\eps}$. Thus
\begin{align*}
    \norm{u^*au\otimes \Id}_\varphi&\leq\norm{(u^*a\otimes \Id)(u\otimes \Id-\Id\otimes \omega_\mcX(u))}_\varphi + \norm{u^*a\otimes \omega_\mcX(u)}_\varphi\\
    &=\norm{(a\otimes \Id)(u\otimes \Id-\Id\otimes \omega_\mcX(u))}_\varphi + \norm{a\otimes \Id}_\varphi\\
    &\leq \norm{a}_{\mcA}\deg(u)\sqrt{\eps} + \norm{a\otimes \Id}_\varphi. 
\end{align*}
Replacing $u$ with $u^*$ and $a$ with $u^* a u$ in this calculation, we see that 
\begin{equation*}
    \norm{a \otimes 1}_{\varphi} \leq \norm{a}_{\mcA} \deg(u)  \sqrt{\eps} + \norm{u^* a u \otimes 1}_{\varphi}, 
\end{equation*}
and hence
$|\norm{u^* a u \otimes 1}_{\varphi} - \norm{a \otimes 1}_{\varphi}| \leq \norm{a}_{\mcA}\deg(u) \sqrt{\eps}$.
\end{proof}

\begin{lemma}\label{lemma:realpart}
    Let $\mcA$ be a $*$-algebra generated by a finite set of unitaries $\mcX$, and let $\varphi$ be an $(\eps,\mcX)$-synchronous state on $\mcA \otimes
    \mcA$. If $u$ is a monomial in $\C^*\ang{\mcX}$ and $a\in\C^*\ang{\mcX}$ is a hermitian square in $\mcA$, then 
    \begin{align*}
        \Re\big(\varphi(u^*aua\otimes\Id)  \big)\geq -   \norm{u^*au}_{1,1}\norm{a\otimes \Id}_\varphi \sqrt{\eps}.
    \end{align*}
\end{lemma}
\begin{proof}
Observe that
    \begin{align*}
         \varphi(u^*aua\otimes\Id)=\varphi\big(  a\otimes \omega_\mcX(u^*au)\big)+\varphi\Big(\big(u^*au\otimes \Id-\Id\otimes \omega_\mcX(u^*au)\big)\big(a\otimes \Id\big)\Big).
    \end{align*}
Since $a=b^*b$ in $\mcA$ for some $b\in\C^*\ang{\mcX}$, $a\otimes \omega_\mcX(u^*au)=\big(b\otimes \omega_\mcX(u^*b^*)  \big)^*\big(b\otimes \omega_\mcX(u^*b^*)  \big)$ is a hermitian square in $\mcA$. Hence $\varphi\big(  a\otimes \omega_\mcX(u^*au)\big)\geq 0$. By  the Cauchy–Schwarz inequality and \Cref{lemma:xy_stability},
    \begin{align*}
        \left\vert\varphi\Big(\big(u^*au\otimes \Id-\Id\otimes \omega_\mcX(u^*au)\big)\big(a\otimes \Id\big)\Big)\right\vert&\leq \norm{u^*au\otimes \Id-\Id\otimes \omega_\mcX(u^*au)}_\varphi \cdotp\norm{a\otimes \Id}_\varphi\\
        &\leq \norm{u^*au}_{1,1}\norm{a\otimes \Id}_\varphi\sqrt{\eps}.
    \end{align*}
Since $\Re(\alpha)\geq -\abs{\alpha}$ for every complex number $\alpha$, the lemma follows.
\end{proof}

\begin{lemma}\label{prop:rbound}
    Let $\mcA$ be a $*$-algebra generated by a finite set of unitaries $\mcX$, and 
    let $\iota : \mcA \to \mcA \otimes \mcA : a \mapsto a \otimes 1$ be the
    left inclusion. Suppose $\mcR \subseteq \C^*\ang{\mcX}$ is a set of $*$-polynomials
    over $\mcX$, and let $\varphi$ be an $(\eps,\mcX)$-synchronous state on $\mcA \otimes \mcA$
    such that $\tau:=\varphi \circ \iota$ is an $(\eps,\mcR)$-state on $\mcA$. If $f \in \mcA$ has an
    $\mcR$-decomposition of size $\Lambda$, then 
    $\norm{f}_{\tau} \leq \Lambda \sqrt{\eps}$.
\end{lemma}
\begin{proof}
    Let $\sum_{i=1}^n\lambda_iu_ir_iv_i$ be an $\mcR$-decomposition for $f$ in
    $\mcA$ of size $\Lambda$.  By \Cref{lemma:xy_stability}, $\norm{v_i\otimes\Id-\Id\otimes
    \omega_\mcX(v_i)}_\varphi\leq deg(v_i)\sqrt{\eps}$ for every $1\leq i\leq n$. Since $\varphi \circ \iota$ is
    an $(\eps,\mcR)$-state on $\mcA$ and $\Id\otimes \omega_\mcX(v_i)$ is unitary in $\mcA\otimes \mcA$, it
    follows that $\norm{r_i\otimes \omega_\mcX(v_i)}_\varphi=\norm{r_i}_{\varphi \circ
    \iota}\leq \sqrt{\eps}$ for all $1\leq i\leq n$. Then
    \begin{align*}
        \norm{f\otimes \Id}_\varphi&=\norm{\sum_{i=1}^n\lambda_iu_ir_iv_i\otimes \Id}_\varphi\leq \sum_{i=1}^n\abs{\lambda_i}\norm{r_iv_i\otimes \Id}_\varphi\\
        &\leq \sum_{i=1}^n\abs{\lambda_i}\big(\norm{(r_i\otimes \Id)(v_i\otimes \Id-\Id\otimes\omega_\mcX(v_i))}_\varphi +\norm{r_i\otimes \omega_\mcX(v_i)}_\varphi \big)\\
        &\leq \sum_{i=1}^n\abs{\lambda_i}\big(\norm{r_i}_{\mcA}\deg(v_i)\sqrt{\eps} +\sqrt{\eps} \big) = \Lambda\sqrt{\eps},
    \end{align*}
 so $\norm{f}_\tau=\norm{f\otimes \Id}_\varphi\leq \Lambda\sqrt{\eps}$.   
\end{proof}

\subsection{Dealing with infinite operator norm}

The results presented so far in this section are useful for algebras $\mcA$
with $\norm{a}_\mcA<\infty$ for all $a\in\mcA$. However, for the free algebras
$\mcA=\C^*\ang{x_1,\cdots,x_n}$ the operator norm $\norm{a}_\mcA= \infty$ for
all $a\in\mcA$ which are not scalar multiples of identity, making the size of
an $\mcR$-decomposition infinite as well. To handle this situation, we show in
this section that for any element $f$ in the free algebra $\mcA$, there is a
set of relations $\mcW(f)$ such that if $\tau$ is a tracial $(\eps,\mcW(f))$-state on $\mcA$, then there is a
tracial state on $\C \Z_2^{*n}$ which is close to $\tau$ on $f$.

\begin{lemma}\label{lemma:stability}
For any $A\in\mcB(\mcH)$, the operator $\widetilde{A}:=\sgn\big( \frac{A+A^*}{2}\big)$ is a unitary involution in $C^*\ang{A}''$. If in addition, $\xi$ is a vector in $\mcH$ satisfying (i) $\norm{(A^2-\I_\mcH)\xi}\leq\epsilon$, (ii) $\norm{((A^*)^2-\I_\mcH)\xi}\leq\epsilon$, (iii) $\norm{(AA^*-\I_\mcH)\xi}\leq\epsilon$, (iv) $\norm{(A^*A-\I_\mcH)\xi}\leq\epsilon$, and (v) $\norm{(A-A^*)\xi}\leq\epsilon$ for some $\eps\geq0$, then $\norm{(A-\widetilde{A})\xi}\leq 2 \epsilon$ and $\norm{(A^*-\widetilde{A})\xi}\leq 2 \epsilon$.
\end{lemma}

Here $\sgn$ denotes the function $\R\to \{-1,1\}$ sending elements in $[0,+\infty)$ to $1$ and elements in $(-\infty,0)$ to $ -1$. 

\begin{proof}
The spectrum of $\tfrac{A+A^*}{2}$ is in $\R$, so by Borel functional calculus, $\widetilde{A}:=\sgn(\tfrac{A+A^*}{2})$ is a unitary involution and is in $C^*\ang{A}''$. 
Since $\sgn$ and $\abs{\cdotp}$ are measurable functions on $\R$ satisfying $(1+\abs{a})(\sgn(a)-a)=\sgn(a)(1-a^2)$ for all $a\in\R$, we obtain that
\begin{align*}
    \norm{(\widetilde{A}-\tfrac{A+A^*}{2})\xi}&=\norm{(\sgn(\tfrac{A+A^*}{2})-\tfrac{A+A^*}{2})\xi}\leq \norm{(\operatorname{Id}_\mcH+\abs{\tfrac{A+A^*}{2}})(\sgn(\tfrac{A+A^*}{2})-\tfrac{A+A^*}{2})\xi}\\
    &=\norm{\sgn(\tfrac{A+A^*}{2})(\operatorname{Id}_\mcH-(\tfrac{A+A^*}{2})^2)\xi} =\norm{(\operatorname{Id}_\mcH-(\tfrac{A+A^*}{2})^2)\xi}\\
    &=\norm{(\operatorname{Id}_\mcH-\tfrac{A^2+(A^*)^2+AA^*+A^*A}{4})\xi}\leq \epsilon,
\end{align*}
where the last inequality follows from the conditions $(i)-(iv)$. By condition $(v)$, $\norm{(A-\tfrac{A+A^*}{2})\xi}=\norm{\tfrac{A-A^*}{2}\xi}\leq \tfrac{\eps}{2}$, so
$\norm{(\widetilde{A}-A)\xi}\leq \norm{(\widetilde{A}-\tfrac{A+A^*}{2})\xi}+\norm{(A-\tfrac{A+A^*}{2})\xi}
    \leq \tfrac{3}{2} \epsilon\leq 2\eps$.
Similarly we have $\norm{(\widetilde{A}-A^*)\xi}\leq 2 \epsilon$.
\end{proof}

\begin{lemma}\label{lemma:tracialfree} 
    Let $\mcX$ be a finite set. There is a computable mapping $\mcW$ from
    $\Q^*\ang{\mcX}$ to finite subsets of $\Q^*\ang{\mcX}$ such that
    \begin{enumerate}[(a)]
        \item $\mcW(f)\subseteq \ker q$ for all $f\in \Q^*\ang{\mcX}$, where
            $q:\C^*\ang{\mcX}\to\C\Z_2^{*\mcX}$ is the quotient map, and

        \item for any tracial state $\tau$ on $\C^*\ang{\mcX}$, there is a
            tracial state $\widehat{\tau}$ on $\C\Z_2^{*\mcX}$ such that 
            for all $f \in \Q^*\ang{\mcX}$ and $\eps \geq 0$, if $\tau$ is a tracial $(\eps,\mcW(f))$-state, then 
            \begin{align*}
                \big\vert \norm{f}_{\widehat{\tau}}-\norm{f}_\tau \big\vert\leq 2\norm{f}_{1,1}\sqrt{\eps}.
            \end{align*}
    \end{enumerate}
\end{lemma}

\begin{proof}
    If $u = x_1 \cdots x_n \in \C^*\ang{\mcX}$ is a monomial of degree $n \geq 0$, 
    then define 
    \begin{align*}
        \mcW(u):=&\left\{(x_ix_i^*-\Id)x_{i+1}\cdots x_n:1\leq i\leq n  \right\}\cup\left\{(x_i^*x_i-\Id)x_{i+1}\cdots x_n:1\leq i\leq n  \right\}\\
        \cup& \left\{(x_i^2-\Id)x_{i+1}\cdots x_n:1\leq i\leq n  \right\}\cup \left\{((x_i^*)^2-\Id)x_{i+1}\cdots x_n:1\leq i\leq n  \right\}\\
        \cup& \left\{(x_i-x_i^*)x_{i+1}\cdots x_n:1\leq i\leq n  \right\}
\end{align*}
    (so in particular, $\mcW(1) = \emptyset$). If $f \in \Q^*\ang{\mcX}$ has support 
    $u_1,\ldots,u_\ell$, or in other words $f=\sum_{i=1}^\ell\lambda_iu_i\in\Q^*\ang{\mcX}$ for distinct $*$-monomials
    $u_1,\ldots,u_\ell$ and non-zero coefficients
    $\lambda_1,\ldots,\lambda_\ell$, we define
    $\mcW(f):=\bigcup_{i=1}^\ell\mcW(u_i)$. Clearly $\mcW$ is computable. If $u$
    is a monomial, then $q(s^* s) = 0$ for all $s \in \mcW(u)$, so the same is
    true for $\mcW(f)$. 

    Let $\big(\pi:\C^*\ang{\mcX}\arr \mcB(\mcH),v\big)$ be the GNS
    representation of $\tau$. Note that $T \mapsto \ang{v,Tv}$ extends to a tracial
    state on $\pi(\C^*\ang{\mcX})''$. For every $x\in\mcX$, let $\widehat{\pi}(x):=\sgn
    \big(\tfrac{\pi(x)+\pi(x)^*}{2}\big) $. By \Cref{lemma:stability},
    $\widehat{\pi}(x)$ is a unitary involution in $C^*\ang{\pi(x)}''$ for every
    $x\in\mcX$, so $\widehat{\pi}$ induces a $*$-representation of
    $\C\Z_2^{*\mcX}$ on $\mcH$ and
    $\widehat{\tau}(a):=\ang{v,\widehat{\pi}(a)v}$ defines a tracial state on
    $\C\Z_2^{*\mcX}$. 

    Suppose $u = x_1 \cdots x_n$ is a monomial.  If  $\tau(s^*s)\leq \eps$ for
    all $s\in\mcW(u)$, then $\norm{\pi(s)v}\leq \sqrt{\eps}$ for all $s\in\mcW(u)$,
    and hence by \Cref{lemma:stability},
    \begin{align*}
        \norm{\big(\widehat{\pi}(x_i)-\pi(x_i)\big)\pi(x_{i+1}\cdots x_n)v}\leq 2\sqrt{\eps}
    \end{align*}
for all $1\leq i\leq n$. It follows that
\begin{align*}
    \norm{\big(\widehat{\pi}(u)-\pi(u)\big)v}&\leq \sum_{i=1}^n\norm{\widehat{\pi}(x_1\cdots x_{i-1})\big(\widehat{\pi}(x_i)-\pi(x_i)\big)\pi(x_{i+1}\cdots x_{n})v}\\
    &=\sum_{i=1}^n\norm{\big(\widehat{\pi}(x_i)-\pi(x_i)\big)\pi(x_{i+1}\cdots x_{n})\psi}\leq 2n\sqrt{\eps}.
\end{align*}

    For a general element $f \in \Q^*\ang{\mcX}$, write $f = \sum_{i=1}^\ell
    \lambda_i u_i$ as above.  If $\tau(s^*s)\leq \eps$ for all $s\in\mcW(f)$, then
    \begin{align*}
         \big\vert \norm{f}_\tau-\norm{f}_{\widehat{\tau}} \big\vert&= 
      \big\vert \norm{\widehat{\pi}(f)v}-\norm{\pi(f)v}\big\vert\leq \norm{\big(\widehat{\pi}(f)-\pi(f)\big)v}\\
      &\leq \sum_{i=1}^\ell \abs{\lambda_i}\cdotp\norm{\big(\widehat{\pi}(u_i)-\pi(u_i)\big)v}  \leq\sum_{i=1}^\ell 2\abs{\lambda_i}\deg(u_i)\sqrt{\eps}=2\norm{f}_{1,1}\sqrt{\eps}.
    \end{align*}
    We conclude that $\mcW$ is a computable mapping satisfying (a) and (b).
\end{proof}

Note that the proof of \Cref{lemma:tracialfree} only requires $\tau(s^*s)\leq \eps$ for all $s\in\mcW(f)$. We use the stronger hypothesis that $\tau$ is an $(\eps,\mcW(f))$-state for convenience. 

\begin{lemma}\label{lemma:free}
    Let $\mcX$ be a finite set, and let $\mcW$ be the map from
    \Cref{lemma:tracialfree}.  Suppose $\varphi$ is an $(\eps,\mcX)$-synchronous
    state on $\C^*\ang{\mcX}\otimes \C^*\ang{\mcX}$ for some $\eps\geq 0$, and
    $\varphi(s^*s\otimes\Id)\leq \eps,\varphi(\Id\otimes s^*s)\leq \eps$ for all
    $s\in\bigcup_{x\in\mcX}\mcW(x)$. Then there is a $(25\eps,\mcX)$-synchronous
    state $\widehat{\varphi}$ on $\C\Z_2^{*\mcX}\otimes\C\Z_2^{*\mcX}$ such that 
    for all $f\in\Q^*\ang{\mcX}$, if $\varphi(s^*s\otimes \Id) \leq \eps$
    for all $s \in \mcW(f)$, then 
    \begin{align*}
        \big\vert \norm{f\otimes\Id}_{\widehat{\varphi}}-\norm{f\otimes\Id}_\varphi\big\vert\leq 2\norm{f}_{1,1}\sqrt{\eps}.
    \end{align*}
\end{lemma}
\begin{proof}
    Let $(\pi:\C^*\ang{\mcX}\otimes \C^*\ang{\mcX}\arr\mcB(\mcH),\psi)$ be the
    GNS representation of $\varphi$. For every $x\in\mcX$, let
    $\widehat{\pi}(x\otimes
    \Id):=\sgn\big(\tfrac{\pi(x\otimes\Id)+\pi(x\otimes\Id)^*}{2}\big)$ and
    $\widehat{\pi}(\Id\otimes x):=\sgn\big(\tfrac{\pi(\Id\otimes x)+\pi(\Id\otimes
    x)^*}{2}\big)$. Then by \Cref{lemma:stability}, for every $x\in\mcX$,
    $\widehat{\pi}(x\otimes \Id)$ and $\widehat{\pi}(\Id\otimes x)$ are unitary
    involutions in $C^*\ang{\pi(x\otimes\Id)}''$ and $C^*\ang{\pi(\Id\otimes
    x)}''$ respectively. So for every $x,x'\in\mcX$, $[\pi(x\otimes
    \Id),\pi(\Id\otimes x')]=0$ implies $[\widehat{\pi}(x\otimes
    \Id),\widehat{\pi}(\Id\otimes x')]=0$, and hence $\widehat{\pi}$ induces a
    $*$-representation of $\C\Z_2^{*\mcX}\otimes\C\Z_2^{*\mcX}$ on $\mcH$ and
    $\widehat{\varphi}(a):=\ang{\psi,\widehat{\pi}(a)\psi}$ defines a state on
    $\C\Z_2^{*\mcX}\otimes\C\Z_2^{*\mcX}$. By the assumption, $\varphi(s^*s\otimes
    \Id)\leq \eps$ and $\varphi(\Id\otimes s^*s)\leq \eps$ for all
    $s\in\bigcup_{x\in\mcX}\mcW(x)$. This implies $\norm{\pi(s\otimes\Id)\psi}\leq
    \sqrt{\eps}$ and $\norm{\pi(\Id\otimes s)\psi}\leq \sqrt{\eps}$ for all
    $s\in\bigcup_{x\in\mcX}\mcW(x)$, and hence by \Cref{lemma:stability},
    $\norm{\big(\widehat{\pi}(x\otimes\Id)-\pi(x\otimes\Id)\big)\psi}\leq
    2\sqrt{\eps}$ and $\norm{\big(\widehat{\pi}(\Id\otimes x)-\pi(\Id\otimes
    x)\big)\psi}\leq 2\sqrt{\eps}$ for all $x\in\mcX$. It follows that
    \begin{align*}
        \norm{\big(\widehat{\pi}(x\otimes\Id)-\widehat{\pi}(\Id\otimes x)\big)\psi}\leq& \norm{\big(\widehat{\pi}(x\otimes\Id)-\pi(x\otimes\Id)\big)\psi}+\norm{\big(\pi(x\otimes \Id)-\pi(\Id\otimes x)\big)\psi}\\
        &+\norm{\big(\widehat{\pi}(\Id\otimes x)-\pi(\Id\otimes x)\big)\psi}\leq 2\sqrt{\eps}+\sqrt{\eps}+2\sqrt{\eps}=5\sqrt{\eps}.
    \end{align*}
    This means $\widehat{\varphi}$ is a $(25\eps,\mcX)$-synchronous state
    $\widehat{\varphi}$ on $\C\Z_2^{*\mcX}\otimes\C\Z_2^{*\mcX}$. The rest of
    the lemma follows from a similar argument to the proof of 
    \Cref{lemma:tracialfree}.
\end{proof}

\section{Construction of a suitable family of algebras}\label{sec:algebracons}

We begin the proofs of Theorems \ref{thm:mainbipartite} and \ref{thm:maintracial},
with the case $\mcA = \Q \Z_2^{*N}$. To explain the idea of the proof,
suppose that $\mcL \subseteq \N$ is an RE set, and that we can find an integer
$N$, a self-adjoint element $w \in \Q\Z_2^{*N}$, and a computable family of finitely
presented algebras $\mcA_{\mcL}(m) = \C^*\ang{ \C \Z_2^{*N}: \mcR_m}$ such that
$\mcA_{\mcL}(m)$ has a tracial state $\tau$ with $\tau(w^2) > 0$ if and only
if $m \in \mcL$. If $\alpha(m) := \sum_{r \in \mcR_m} r^* r - w^2$, then a
tracial state $\tau$ on $\mcA_{\mcL}(m)$ with $\tau(w^2) > 0$ pulls back to a
tracial state on $\C \Z_2^{*N}$ with $\tau(\alpha(m)) < 0$. Unfortunately this
idea runs into a \textit{stability problem}: it is not necessarily the case
that every tracial state on $\C \Z_2^{*N}$ with $\tau(\alpha(m))<0$ comes from
a tracial state on $\mcA_{\mcL}(m)$. Indeed, it might be possible to find a
tracial state $\tau$ for which $\tau\left(\sum_{r \in \mcR_m} r^* r\right)$ is small, but
$\tau(w^2)$ is large. 

The main technical idea of this paper is that we can overcome the stability
problem with a carefully chosen algebra $\mcA_{\mcL}(m)$. Specifically,
we start with the algebra generated by a family of positive contractions
$\{\wtd{P}_0, \wtd{P}_1,\ldots\}$ and a family of unitary involutions
$\{\wtd{X}_0,\wtd{X}_1,\ldots\}$ satisfying the relations
\begin{equation}\label{eq:keyrelation}
    \wtd{P}_{n+1} = \wtd{P}_{n} + \wtd{X}_n \wtd{P}_n \wtd{X}_n
\end{equation}
for all $n$ such that $M_{\mcL}$ does not halt on input $m$ after $n$ steps,
where $M_{\mcL}$ is a Turing machine recognizing $\mcL$. This algebra is not
finitely presented, so the purpose of this section is to show that this algebra
can be embedded in a finitely presented algebra $\mcA_{\mcL}(m)$, in such a way
that the size of relation \eqref{eq:keyrelation} is polynomial in $n$. Relation
\eqref{eq:keyrelation} implies that $\tau(\wtd{P}_{n}) = \tfrac{1}{2}
\tau(\wtd{P}_{n+1})$ for any tracial state $\tau$ on $\mcA_{\mcL}(m)$, and
in the next section we use this fact to overcome the stability problem and
finish the proofs of Theorems \ref{thm:mainbipartite} and \ref{thm:maintracial}
(for the case $\mcA = \Q \Z_2^{*N}$) using the strategy outlined in the previous
paragraph.

\begin{theorem}\label{thm:main}
    Let $\mcL\subset \N$ be an RE set, and let $M_\mcL$ be a Turing machine that recognizes $\mcL$. Then there is a 
    finite set $\mcX$, two sequences $\{\wtd{P}_0,\wtd{P}_1, ... \}$ and
    $\{\wtd{X}_0,\wtd{X}_1, ...\}$ in $\Q^*\ang{\mcX}$, and a computable mapping
    $m\mapsto \mcR_m $ from  $\N$ to finite subsets of $\Q\Z_2^{*\mcX}$, such that 
    \begin{enumerate}[(a)]
        \item for all $n \geq 0$, $\wtd{P}_n$ is a hermitian square in $\Q
            \Z_2^{*\mcX}$ with $ \norm{\wtd{P}_n}_1 =1$ and $\norm{\wtd{P}_n}_{1,1} = 3(n+1)/2$, and
            $\wtd{X}_n$ is a monomial in $\Q^*\ang{\mcX}$ of degree $4n+1$ such that $\wtd{X}_n^2=1$ in $\Q \Z_2^{*\mcX}$,
        \item there are constants $C>0,k\in\N$ such that for every $m\in\N$ and
            $n \geq 0$, if $M_\mcL$ hasn't halted after $n$ steps on input $m$, then
            $ \wtd{P}_n + \wtd{X}_n \wtd{P}_n \wtd{X}_n-\wtd{P}_{n+1}$ is trivial in
            $\mcA_{\mcL}(m) := \C^*\ang{\C\Z_2^{*\mcX}:\mcR_m}$ and has an
            $\mcR_m$-decomposition in $\C \Z_2^{\mcX}$ of size $\leq C \big((n+1)m\big)^k$,
        \item if $m\in \mcL$, then there is a tracial state $\tau$ on $\mcA_\mcL(m)$ such that $\tau(\wtd{P}_0^2)>0$, and 
        \item there is a positive integer $D$ such that $\norm{r}_{1,1}\leq Dm$ for all $r\in\mcR_m,m\in\N$.
    \end{enumerate} 
\end{theorem}
In part (a), $\norm{\cdot}_{1}$ and $\norm{\cdot}_{1,1}$ denote the 1-norm and Sobolev seminorm on $\C^*\ang{\mcX}$ defined in the previous section. Since monomials in $\C^*\ang{\mcX}$ are unitary in $\C\Z_2^{*\mcX}$, $\norm{a}_{\C\Z_2^{*\mcX}}\leq \norm{a}_1$ for all $a\in\C^*\ang{\mcX}$. So part (a) implies that $\wtd{P}_n$ is a contraction in $\C\Z_2^{*\mcX}$, i.e.~$\norm{\wtd{P}_n}_{\C\Z_2^{*\mcX}}\leq 1$, and  $\wtd{X}_n$ is a unitary involution in $\C\Z_2^{*\mcX}$ for all $n\geq 0$.

For the rest of this section, we fix an RE set $\mcL\subset\N$, and a deterministic Turing
machine $M_\mcL$ that recognizes $\mcL$. Define $h:\Z \arr \N\cup\{\infty\}$ by
\begin{equation*}
    h(m)= \begin{cases}
    n & m\in\mcL \text{ and } M_\mcL \text{ halts at the } n\text{th step on the input }m  \\
    \infty & m\notin \mcL
    \end{cases},
\end{equation*}
so $h|_\N$ is the running time of $M_\mcL$. Recall that in our convention, all
Turing machines take at least one step, so $h(m) \geq 1$ for all $m \in \Z$. To
prove \Cref{thm:main}, we start in \Cref{SS:groupstuff} by encoding the Turing
machine $M_{\mcL}$ in certain group relations. We finish the proof by constructing
$\mcA_{\mcL}(m)$ in \Cref{ss:construction}.

\subsection{Group stage of construction}\label{SS:groupstuff}
In the first
stage of constructing $\mcA_\mcL(m)$, we build up a group $G_\mcL$ out of amalgamations, central products, and
HNN extensions. We show that $G_{\mcL}$ has a finite generating set, and that
the word problem with respect to this generating set is decidable in polynomial
time. This allows us to embed $G_{\mcL}$ in a finitely presented group $H_S$,
which we use as the basis for defining the algebras $\mcA_{\mcL}(m)$.

For the construction of $G_{\mcL}$, let $[A,B]:=ABA^{-1}B^{-1}$ denote the
group commutator, and let $\Z_n$ be the additive cyclic group of order $n \in\N
\cup \{+\infty\}$ (so in particular, $\Z_{\infty} = \Z$). To order the elements
of $\Z_n$, we represent elements of $\Z_n$ by integers 
    $-\lfloor \tfrac{n+2}{2} \rfloor < j < \lfloor \tfrac{n+1}{2} \rfloor$.
For instance, for $\Z_4$ we use representatives 
$-2,-1,0,1$. This choice of representatives 
has the property that every congruence class is represented by
an element that is minimal in absolute value, i.e.~if $k'$ is the representative
of $[k]$, then $|k'| \leq |k|$.  Also, note that this set of representatives
contains $-1$ for every $n \geq 2$. When $n=+\infty$, elements of $\Z_n = \Z$
represent themselves.

For each $m\in\Z$ and $i\in\Z_{h(m)+1}$, let $K_\mcL(m,i)$ be the group
generated by indeterminates $J,x_{mi},z_{mi}$, subject to the
relations: 
\begin{itemize}
    \item $J^2=x_{mi}^2=z_{mi}^2=\Id$,
    \item $[x_{mi},J] = [z_{mi},J] = 1$, and
    \item  $[x_{mi},z_{mi}] = J^a$, where $a=1$ if $i{+}1 \neq 0$ in $\Z_{h(m)+1}$, 
        and $a=0$ otherwise. 
\end{itemize}
If $i+1\neq 0$ in $\Z_{h(m)+1}$, then the map $x_{mi} \mapsto a$, $z_{mi}
\mapsto b$, $J \mapsto abab$ gives an isomorphism from $K_{S}(m,i)$ to the rank
4 dihedral group 
\begin{equation*}
    D_4=\ang{a,b: a^2 = b^2 = (ab)^4=1 },
\end{equation*}
while if $i+1=0$ in $\Z_{h(m)+1}$ then all generators commute. Hence 
\begin{equation*}
    K_\mcL(m,i) = \begin{cases} D_4 & i+1 \neq 0 \text{ in } \Z_{h(m)+1} \\
                        \Z_2 \times \Z_2 \times \Z_2 & i+1 = 0 \text{ in } \Z_{h(m)+1}
                    \end{cases}.
\end{equation*}
In particular, $J$ is central and $\ang{J} = \Z_2$ in both cases. 
For $m \in \Z$, let
\begin{align*}
    K_\mcL(m):=\sideset{_{\ang{J}}}{_{i\in\Z_{h(m)+1}}} \prod K_\mcL(m,i)
\end{align*}
be the central product of the groups $K_\mcL(m,i)$, $i \in \Z_{h(m)+1}$, over
the central subgroup $\ang{J} = \Z_2$, and let 
\begin{align*}
    K_\mcL:={}_{\ang{J}} \Asterisk_{m\in\Z} K_\mcL(m)
\end{align*}
be the amalgamated free product of the groups $K_\mcL(m)$, $m \in \Z$, over $\ang{J} = \Z_2$.
In other words, $K_\mcL$ is generated by indeterminates 
\begin{equation*}
    \mcX_{K} := \{J,x_{mi},z_{mi},m\in\Z,i\in\Z_{h(m)+1}\}, 
\end{equation*}
subject to relations
\begin{enumerate}[(K1)]
    \setcounter{enumi}{-1}
    \item $J^2=x_{mi}^2=z_{mi}^2=1$ for all $m\in\Z,i\in\Z_{h(m)+1}$,
    \item $[x_{mi},J] = [z_{mi},J] = 1$ for all $m \in \Z$, $i \in \Z_{h(m)+1}$,
    \item for all $m\in\Z,i \in \Z_{h(m)+1}$, 
        \begin{equation*}
            [x_{mi},z_{mi}] = \begin{cases} J & i+1 \neq 0 \text{ in } \Z_{h(m)+1} \\
                                            \Id & i+1 = 0 \text{ in } \Z_{h(m)+1} 
                              \end{cases}, \text{ and }
        \end{equation*} 
    \item for all $m\in\Z$ and $i \neq j \in \Z_{h(m)+1}$,
        \begin{equation*}
            [x_{mi},z_{mj}] = [x_{mi},x_{mj}] = [z_{mi},z_{mj}] = \Id.
        \end{equation*}
\end{enumerate} 
Note from relations (K1) that $J$ is central in $K_{\mcL}$. For each $m \in \Z$
and $i \in \Z_{h(m)+1}$, the elements $\{1,x_{mi},z_{mi},x_{mi}z_{mi}\}$ form
a set of right coset representatives for $\ang{J}$. By the normal form theorem
for central products, every element $w \in K_{\mcL}(m)$, $m \in \Z$ 
can be written uniquely as
\begin{equation}\label{KSmnormalform}
    w = J^cx_{m,i_1}^{a_1}z_{m,i_1}^{b_1}x_{m,i_2}^{a_2}z_{m,i_2}^{b_2}\cdots
x_{m,i_k}^{a_k}z_{m,i_k}^{b_k},
\end{equation}
where $k \geq 0$, $i_1 < i_2 < \ldots < i_k$ 
are representatives of elements from $\Z_{h(m)+1}$ as described above, $c \in
\Z_2$, and $(a_j,b_j) \in \{(0,1),(1,0),(1,1)\} \subset \Z_2 \times \Z_2$ for
all $1 \leq j \leq k$. The function $\eta_m(w) =
J^cx_{m,i_1}^{a_1}z_{m,i_1}^{b_1}\cdots x_{m,i_k}^{a_k}z_{m,i_k}^{b_k}$ is a
normal form for $K_{\mcL}(m)$. Note that we take the indices $i_1,\ldots,i_k$
to be from the chosen set of representatives for $\Z_{h(m)+1}$, rather than 
elements of $\Z_{h(m)+1}$, so $\eta(w)$ is a word over
\begin{equation*}
    \widetilde{\mcX}_K := \left\{J, x_{mi}, z_{mi}, m \in \Z, 
    -\lfloor \tfrac{h(m)+3}{2}\rfloor < i < \lfloor \tfrac{h(m)+2}{2} \rfloor \right\}
\end{equation*}
rather than $\mcX_K$. We use this convention so that the indices are ordered,
and so we can encode words over the generators as binary strings. 
For convenience in what follows, we assume that words over $\widetilde{\mcX}_K
\cup \widetilde{\mcX}_{K}^{-1}$ are encoded in a fairly typical way as a
sequence of substrings representing symbols and indices, where symbols are
encoded with a fixed number of bits, and indices are written in binary with a
sign bit.  For instance, we might encode $x^{-1}_{m,i}$ by the bit sequence
for $x^{-1}$ followed by the binary representations of $m$ and $i$ separated
by delimiters, and $x^{-3}_{m,i} z_{m',i'}$ by this same sequence repeated three times (we
don't represent exponents in binary) followed by the encoding of $z_{m',i'}$. 
The \textbf{bit-length} (as opposed to the usual length) of a word over
$\widetilde{\mcX}_K \cup \widetilde{\mcX}_{K}^{-1}$ is the length of this
binary encoding. The bit-length of a word $w$ depends on the indices $m$ and
$i$ of generators $x_{mi}^{\pm 1}$ and $z_{mi}^{\pm 1}$ appearing in $w$. In
what follows, we need to pay special attention to how large $i$ gets, so we
define the \textbf{maximum index} of a word $w$ over $\widetilde{\mcX}_{K} \cup
\widetilde{\mcX}_{K}^{-1}$ to be
\begin{equation*}
    \max \{r \geq 0: \text{one of }x_{m,\pm r}^{\pm 1}, z_{m,\pm r}^{\pm 1} \text{ appears in 
    } w \text{ for some } m \in \Z \} \cup \{0\}.
\end{equation*}

Returning to normal forms, the subset 
\begin{align*}
    \mcC(m):=\big\{x_{m,i_1}^{a_1}z_{m,i_1}^{b_1}\cdots x_{m,i_k}^{a_k}z_{m,i_k}^{b_k}:\ & 
        k \geq 0, -\lfloor \tfrac{h(m)+3}{2} \rfloor < i_1 < \ldots < i_k < \lfloor \tfrac{h(m)+2}{2} \rfloor, \\*
        & (a_j,b_j) \in \{(0,1),(1,0),(1,1)\}, 1 \leq j \leq k\big\}
\end{align*}
is a set of right coset representatives for $\ang{J}$ in $K_{\mcL}(m)$. 
By the normal form theorem for amalgamated free products, every element $w \in
K_{\mcL}$ can be written uniquely as 
\begin{equation}\label{KSnormalform}
    w = J^c w_{1} \cdots w_{\ell}
\end{equation}
where $\ell \geq 0$, $c \in \Z_2$, and $w_j \in \mcC(m_j) \setminus \{1\}$
for some integers $m_1,\ldots,m_\ell$ such that $m_j \neq m_{j+1}$ for all $1
\leq j \leq \ell-1$. The function $\eta(w) = J^c w_1 \cdots
w_\ell$ is a normal form for $K_{\mcL}$. This normal form can be computed efficiently:
\begin{proposition}\label{prop:normalform}
    Let $\eta$ be the normal form for $K_S$ described above.
    If $w$ is given as a word over $\widetilde{X}_K \cup \widetilde{X}_{K}^{-1}$,
    then $\eta(w)$ is computable in polynomial time in the bit-length of $w$.
    If we write $\eta(w) = J^c w_1 \cdots w_\ell$ as in \cref{KSnormalform}, then
    the length, bit-length, and maximum index of $w_1 \cdots w_{\ell}$ are at
    most the length, bit-length, and maximum index of $w$ respectively. 
\end{proposition}
\begin{proof}
    To compute $\eta(w)$, suppose first that $w$ is a word over $\widetilde{\mcX}_K(m) \cup \widetilde{\mcX}_K(m)^{-1}$,
    where
    \begin{equation*}
        \widetilde{\mcX}_K(m) := \left\{J, x_{mi}, z_{mi}, -\lfloor
        \tfrac{h(m)+3}{2} \rfloor < i < \lfloor \tfrac{h(m)+2}{2} \rfloor \right\}.
    \end{equation*} 
    Note that if $-\lfloor \tfrac{h(m)+3}{2} \rfloor < i < \lfloor \tfrac{h(m)+2}{2} \rfloor$,
    then $i+1 = 0$ in $\Z_{h(m)+1}$ if and only if $i=-1$. Using relations (K1)-(K3), we can
    reorder $w$ so that it can be written as in \cref{KSmnormalform} for some
    $c, a_1,b_1,\ldots,a_k,b_k \in \Z$ (rather than $\Z_2$). Using relations (K0), we can then
    reduce $c, a_1,b_1,\ldots,a_k,b_k$ mod $2$ to put $w$
    in normal form $\eta_m(w) = J^c \overline{w}$, where $\overline{w} \in
    \mcC(m)$. Since a list of integers can be sorted in polynomial time, this
    process takes polynomial time in the bit-length of $w$. Since every generator in $\overline{w}$
    occurs in $w$, and no generator appears in $\overline{w}$ more than once, the length, bit-length,
    and maximum index of $\overline{w}$ are at most the length, bit-length, and maximum index of $w$
    respectively. 

    If $w$ is a word over $\widetilde{\mcX}_K \cup \widetilde{\mcX}_K^{-1}$, then we
    can write $w = J^c w_1 \cdots w_{\ell}$, where $c \in \Z_2$, $w_j$ is a word over
    $\widetilde{\mcX}_K(m_j) \cup \widetilde{\mcX}_K(m_j)^{-1}$ for $1 \leq i \leq \ell$, and  $m_1,\ldots,m_{\ell}$ is
    some sequence of integers such that $m_{j} \neq m_{j+1}$ for $1 \leq j \leq \ell-1$. Using
    the procedure in the last paragraph, we can put $w_j$ in normal form $J^{c_j} \overline{w}_j$
    for $K_{\mcL}(m_j)$, where $\overline{w}_j \in \mcC(m_j)$ and $c_j \in \Z_2$. The
    length of $\overline{w}_j$ is at most the length of $w_j$, so the length of $\overline{w}_1
    \cdots \overline{w}_{\ell}$ is at most the length of $w_1 \cdots w_{\ell}$, and the same
    is true of bit-length and maximum index. Let $c' = c+ \sum_j c_j$. If
    $\overline{w}_j \neq 1$ for all $1 \leq j \leq \ell$, then $\eta(w) =
    J^{c'} \overline{w}_1 \cdots \overline{w}_{\ell}$. Otherwise, if
    $\overline{w}_j = 1$ for some $1 \leq j \leq \ell$, then we can repeat this
    procedure with $w' = J^{c'} \overline{w}_1 \cdots \overline{w}_{j-1}
    \overline{w}_{j+1} \cdots \overline{w}_\ell$. Since $\ell$ decreases in each iteration
    of this algorithm, the overall running time will be polynomial in the bit-length of $w$.
    Also, since the length of $w_1 \cdots w_{\ell}$ does not increase in each iteration, the 
    length of $\overline{w}_1 \cdots \overline{w}_\ell$ in the final iteration is at most the
    length of the initial input $w$, and the same is true of bit-length and maximum index.
\end{proof}
\begin{remark}\label{rmk:noJs}
    Since, for instance, $\eta(z_{m0} x_{m0}) = J x_{m0} z_{m0}$, the length
    and bit-length of $\eta(w)$ is not necessarily less than the length and
    bit-length of $w$. 
    Proposition \ref{prop:normalform} does imply that the length of
    $\eta(w)$ is at most one more than the length of $w$.  Since the generator
    $J$ does not contribute to the maximum index, the maximum index of
    $\eta(w)$ is at most the maximum index of $w$.
\end{remark}
\begin{remark}\label{rmk:representative}
The above algorithm for computing $\eta$ assumes that the generators $x_{mi}$
and $z_{mi}$ in the input are given by specifying the representative of the index $i \in \Z_{h(m)+1}$.
The algorithm does not work as stated without this assumption. For instance, on $w = x_{m,-3} \cdot x_{m,7}$ the
algorithm will return $w$ back again, but if $h(m)=9$ then $x_{m,7} = x_{m,-3}$ in $K_{\mcL}$, and $\eta(w) = 1$.
If $i \in \Z$, then the representative for $i$ in $\Z_{h(m)+1}$ can be computed
by running the Turing machine $M_{\mcL}$ on input $m$ for $2|i|-1$ steps. If
$M_{\mcL}$ halts by this point then we know $h(m)$ and can replace $i$ by its
representative if necessary, while if it does not halt then $h(m) > 2 |i| - 1$
and hence $-\lfloor \frac{h(m)+3}{2} \rfloor < i < \lfloor \frac{h(m)+2}{2}
\rfloor$. The running time of this procedure is polynomial in the bit-length
of $m$ and the value of $i$, but not polynomial in the bit-length of $i$. We
will use this procedure for computing the representative of $i$ later.  
\end{remark}

For the final step of constructing $ G_\mcL$, we take iterated HNN extensions over several subgroups of $K_\mcL$:
\begin{align*}
    K_{\mcL,x}&:=\ang{J,x_{mi},m\in\Z,i\in\Z_{h(m)+1}},\\
    K_{\mcL,z}&:=\ang{J,z_{mi},m\in\Z,i\in\Z_{h(m)+1}}, \text{ and}\\
    K_{\mcL,0}&:=\ang{J,x_{m,0},z_{m,0},m\in\Z}.
\end{align*}

\begin{lemma}\label{lemma:subgroups}\ 
\begin{enumerate}[(a)]

    \item An element $w\in K_\mcL$ is in $ K_{\mcL,x}$ if and only if its
        normal form $\eta(w)$ is a word over $\{J, x_{mi},\ m \in \Z, -\lfloor
        \tfrac{h(m)+3}{2}\rfloor < i < \lfloor \tfrac{h(m)+2}{2} \rfloor \}$. In
        addition, there is an automorphism $\phi_S:K_{\mcL,x}\arr K_{\mcL,x}$ sending
        $J\mapsto J$ and $x_{mi}\mapsto x_{m,i+1}$ for all $m\in\Z,i\in\Z_{h(m)+1}$.
    
    \item An element $w\in K_\mcL$ is in $ K_{\mcL,z}$ if and only if its
        normal form $\eta(w)$ is a word over $\{J,z_{mi},\ m \in \Z, -\lfloor
        \tfrac{h(m)+3}{2}\rfloor < i < \lfloor \tfrac{h(m)+2}{2} \rfloor \}$. In
        addition, there is an automorphism $\phi_T:K_{\mcL,z}\arr K_{\mcL,z}$ sending
        $J\mapsto J$ and \,  $z_{mi}\mapsto z_{m,i+1}$ for all
        $m\in\Z,i\in\Z_{h(m)+1}$.
    
    \item An element $w\in K_\mcL$ is in $ K_{\mcL,0}$ if and only if its
        normal form $\eta(w)$ is a word over $\{J,x_{m,0},z_{m,0},m \in \Z\}$. In
        addition, there is an automorphism $\phi_W:K_{\mcL,0}\arr K_{\mcL,0}$ sending
        $J\mapsto J$, $x_{m,0}\mapsto x_{m+1,0}$, and $z_{m,0}\mapsto z_{m+1,0}$ for
        all $m\in\Z$.
\end{enumerate}
\end{lemma}
\begin{proof}
For part (a), observe that the subgroup
$K_{\mcL,x}(m,i):=\ang{J,x_{mi}}\subseteq K_\mcL(m,i)$ is isomorphic to $\Z_2
\times \Z_2$ for all $m\in\Z,i\in\Z_{h(m)+1}$. For $m\in \Z$, let
$K_{\mcL,x}(m):=\ang{J,x_{mi}:i\in\Z_{h(m)+1}}\subseteq K_{\mcL}(m)$, and
consider the natural surjective homomorphism 
\begin{align*}
    \Phi_m:\sideset{_{\ang{J}}}{_{i\in\Z_{h(m)+1}}} \prod K_{\mcL,x}(m,i)\arr K_{\mcL,x}(m).
\end{align*}
By the normal form theorem for central products, every element $w$ in the
domain of $\Phi_m$ can be written uniquely as $ w = J^cx_{m,i_1}x_{m,i_2}\cdots
x_{m,i_k}$, where $k \geq 0$, $-\lfloor \tfrac{h(m)+3}{2}\rfloor <
i_1<\cdots<i_k < \lfloor \tfrac{h(m)+2}{2} \rfloor$, and $c \in \Z_2$.
Comparing with the normal form $\eta_m$ for $K_{\mcL}(m)$, we see that $\Phi_m$
sends normal forms to normal forms. It follows that $\Phi_m$ is an isomorphism,
and that if $w\in K_{\mcL,x}(m)$, then the normal form $\eta_m(w)$ in
$K_\mcL(m)$ is a word containing only $J$ and $x_{mi}$'s.

A similar argument using the normal form theorem for amalgamated free products shows that
 \begin{align*}
     K_{\mcL,x} = {}_{\ang{J}} \Asterisk_{m\in\Z} K_{\mcL,x}(m),
 \end{align*}
and the normal form $\eta(w)$ of an element $w\in K_{\mcL,x}$ is a word over
$\{J,x_{mi}:m\in\Z,-\lfloor \tfrac{h(m)+3}{2}\rfloor < i < \lfloor
\tfrac{h(m)+2}{2} \rfloor \}$.

For the automorphism, observe that 
\begin{align*}
    \sideset{_{\ang{J}}}{_{i\in\Z_{h(m)+1}}} \prod K_{\mcL,x}(m,i)\cong \Z_2\times \Z_2^{\oplus h(m)+1},
\end{align*}
where the first $\Z_2$ is generated by $J$, and $\Z_2^{\oplus h(m)+1}$ is
generated by $\{x_{mi}:i\in \Z_{h(m)+1}\}$. Hence for every $m\in \Z$ there is
an automorphism $\phi_S^{(m)}:K_{\mcL,x}(m)\arr K_{\mcL,x}(m)$ sending
$J\mapsto J$ and $x_{mi}\mapsto x_{m,i+1}$ for all $i\in \Z_{h(m)+1}$, and
these extend to the desired automorphism $\phi_S:K_{\mcL,x}\arr K_{\mcL,x}$.

The argument for part (b) is identical. For part (c), observe that $K_\mcL(m,0)\cong D_4$ for all $m\in \Z$ (since $h(m)\geq 1$), and that $K_{\mcL,0}={}_{\ang{J}} \Asterisk_{m\in\Z} K_{\mcL}(m,0)$ by the normal form theorem for amalgamated free products again. Hence for every $m\in \Z$ there is an isomorphism $\phi_W^{(m)}:K_\mcL(m,0)\arr K_\mcL(m+1,0)$ sending $J\mapsto J$, $x_{m,0}\mapsto x_{m+1,0}$, and $z_{m,0}\mapsto z_{m+1,0}$, and these extend to the desired automorphism $\phi_W:K_{\mcL,0}\arr K_{\mcL,0}$. 
\end{proof}

We can now define $ G_\mcL$ as the iterated HNN extensions of $K_\mcL$ by the automorphisms $\phi_S,\phi_T$, and $\phi_W$ from \Cref{lemma:subgroups}, using variables $S,T$, and $W$ respectively. This means that
\begin{align*}
     G_\mcL=\ang{K_\mcL,S,T,W:&\  SJS^{-1}=TJT^{-1}=WJW^{-1}=J,\\
    & S x_{mi} S^{-1} = x_{m,i+1}, T z_{mi} T^{-1} = z_{m,i+1}, m \in \Z, i \in\Z_{h(m)+1},\\
    & W x_{m,0} W^{-1} = x_{m+1,0}, W z_{m,0}  W^{-1} = z_{m+1,0}, m \in \Z}.
\end{align*}
Unlike $K_\mcL$, this group is finitely generated:
\begin{lemma}\label{lemma:finitelygenerated}
$G_\mcL$ is generated by $\{ J,S,T,W,x_{00},z_{00}\} $.
\end{lemma}
\begin{proof}
For any $m\in\Z,i\in\Z_{h(m)+1}$, we have $x_{mi}=S^{i}W^mx_{00}W^{-m}S^{-i}$ and $z_{mi}=T^{i}W^mz_{00}W^{-m}T^{-i}$ in $G_\mcL$.
\end{proof}
We use this generating set frequently going forward. Hence for notational
convenience, we let $X:=x_{00}$, $Z:=z_{00}$, and $\mcX_G:=\{J,S,T,W,X,Z\}$.
Also, we let $X_{mi}$ and $Z_{mi}$ denote the words
\begin{equation*}
    X_{mi} := S^iW^m X W^{-m}S^{-i} \text{ and }Z_{mi} :=T^i W^mZW^{-m}T^{-i},m,i\in \Z.
\end{equation*} 
The words $X_{mi}$ and $Z_{mi}$ both have length $2|m|+2|i|+1$.

If we restate the presentation of $G_\mcL$ with respect to this generating set, we see that $G_\mcL$ is the group generated by $\mcX_G$, subject to the following relations: 
\begin{enumerate}[(G1)]
    \setcounter{enumi}{-1}
    \item $J^2 = X^2 = Z^2 = 1$. 
    \item $[S,J]=[T,J]=[W,J]=[X,J]=[Z,J]=\Id$, i.e., $J$ is a central element.
    \item \label{anti} For every $m,i \in \Z$, 
            \begin{equation*}
            [X_{mi},Z_{mi}]=
                \begin{cases} J & i+1 \neq 0 \text{ mod } h(m)+1 \\
                              \Id & i+1 = 0 \text{ mod } h(m)+1
                \end{cases}.
            \end{equation*}
        \item\label{Gcommute} For all $m \in \Z$ and $i \neq j$ mod $h(m)+1$,
            \begin{equation*}
                [X_{mi},Z_{mj}]
                = [X_{mi},X_{mj}]
                = [Z_{mi},Z_{mj}] = \Id.
            \end{equation*}
    \item\label{ll:Ghalt} For every $m \in \N$ with $h(m) < +\infty$, 
            \begin{equation*}
                X_{m,h(m)+1} = X_{m,0} \text{ and }
                Z_{m,h(m)+1} = Z_{m,0}.
            \end{equation*}
\end{enumerate}

In the next stage of our construction, we want to use \cref{thm:BORS} to embed $G_\mcL$ into a finitely presented group with polynomial isoperimetric function. For this, we need:

\begin{proposition}\label{prop:polyword}
The word problem for $G_\mcL$ with respect to $\mcX_G$ is solvable in deterministic polynomial time.
\end{proposition}
\begin{proof}
First consider the larger generating set
$\widehat{\mcX}_G:=\widetilde{\mcX}_K\cup\{S,T,W\}$ where
$\widetilde{\mcX}_K=\left\{J,x_{mi},z_{mi},m\in \Z, -\lfloor
\tfrac{h(m)+3}{2}\rfloor < i < \lfloor \tfrac{h(m)+2}{2} \rfloor \right\}$ as
in \cref{prop:normalform}. The word problem with respect to $\widehat{\mcX}_G$
is solvable using Britton's Lemma as follows:
\begin{enumerate}[(1)]
    \item Given a word $w$  over $\widehat{\mcX}_G\cup\widehat{\mcX}_G^{-1}$, write
        \begin{align*}\label{wordoverhatX}
            w=g_0t_1^{e_1}g_1t_2^{e_2}\cdots t_n^{e_n}g_n,
        \end{align*}
        where $ e_1, \ldots , e_n \in \{\pm 1\} $, $t_1,\ldots,t_n\in\{S,T,W\}$, and
        $g_0,\ldots,g_n$ are words over $\widetilde{\mcX}_K\cup \widetilde{\mcX}_K^{-1}$. 

    \item For every $0\leq i\leq n$, compute the normal form $\eta(g_i)$ of
        $g_i$ from \cref{prop:normalform}. 

    \item If $n=0$, then $w=\Id$ in $G_\mcL$ if and only if $\eta(g_0)=\Id$, and we can stop.
        
    \item Otherwise, use \cref{lemma:subgroups} to check whether there is an
        index $1\leq j<n$ such that $e_{j}=-e_{j+1}$, $t_{j}=t_{j+1}=S$, and $g_j\in K_{\mcL,x}$. 
        If such a $j$ exists, then let $g_j'=J^cx_{m_1,i_1'}\cdots x_{m_k,i_k'}$ where
        $\eta(g_j)=J^cx_{m_1,i_1}\cdots x_{m_k,i_k}$ and $-\lfloor
        \tfrac{h(m_l)+3}{2}\rfloor < i_l' < \lfloor \tfrac{h(m_l)+2}{2}\rfloor$
        is the representative of $[i_l+e_j]\in\Z_{h(m_l)+1}$ 
        (computed as in \cref{rmk:representative}). Return to step (1) with
        \begin{align*}
            w=g_0t_1^{e_1}\cdots t_{j-1}^{e_{j-1}}(g_{j-1}g_j'g_{j+1})t_{j+2}^{e_{j+2}}g_{j+2}\cdots t_n^{e_n}g_n.
        \end{align*}
    \item If no such $j$ exists, perform step (4) with $S,K_{\mcL,x}$
        replaced by $T,K_{\mcL,z}$ and $W,K_{\mcL,0}$, with $g_j'$ adjusted to match the automorphisms
        $\phi_T$ and $\phi_W$ from Lemma \ref{lemma:subgroups} respectively.

    \item If no such $j$ exists for $S,T$ and $W$, then $w\neq \Id$ in $G_\mcL$ by Britton's Lemma, 
        and we can stop.
\end{enumerate}

Define the bit-length and maximum index of a word $w$ over
$\widehat{\mcX}_G\cup\widehat{\mcX}_G^{-1}$ analogously to words over
$\widetilde{\mcX}_K \cup \widetilde{\mcX}_K^{-1}$. Specifically, if if $w$ is
written as in step (1), then define the maximum index of $w$ to be $\max_i
I_i$, where $I_i$ is the maximum index of $g_i$.  
Let $N$, $B$, and $I$ be the length, bit-length, and maximum index of the
initial input to the above algorithm.  Since $n$ decreases by two after each
iteration, and $n$ is bounded by $N$, the algorithm is guaranteed to halt after
at most $\lceil N/2 \rceil$ iterations.  In steps (4) and (5), the length of
$g_j'$ is equal to the length of $\eta(g_j)$, which is at most one more than
the length of $g_j$.  The generators $t_j$ and $t_{j+1}$ are removed from $w$
in this step, and $w$ does not change in any other step, so the length of $w$
also decreases in each iteration. The maximum index can increase by one in
steps (4) and (5) when the automorphisms $\phi_S$ and $\phi_T$ are applied, so
the maximum index throughout is bounded by $I+\lceil N/2 \rceil$. Applying automorphisms
$\phi_S$, $\phi_T$, and $\phi_W$ in steps (4) and (5) can also increase the
bit-length of each generator by one (note that the bit-length of the
representative of $[i] \in \Z_{h(m)+1}$ is at most the bit-length of $i$), so
the overall bit-length of $w$ can go up by at most $N$ in each step. Thus the
bit-length of $w$ is bounded by $B + N \cdot \lceil N/2 \rceil$ throughout.

By \cref{prop:normalform}, each normal form $\eta(g_j)$ in step (2) can be
computed in time polynomial in the bit-length of $w$, and we compute at most
$N$ normal forms in this step. In steps (4) and (5), each representative 
$i_\ell'$ can be computed in time polynomial in the bit-length and maximum
index of $w$, and we have to compute at most $N$ representatives. The other
steps can be done in time polynomial in the bit-length of $w$, so we conclude
that the entire algorithm takes time polynomial in $N$, $B$, and $I$. 

Since $I$ is not polynomial in $N$ and $B$ in general, the above algorithm
doesn't run in polynomial time in $N$ or $B$. However, if the initial input
is a word in $\mcX_G \cup \mcX_G^{-1}$, then $I=0$, and $B = CN$ where
$C$ is the number of bits needed to encode each element of $\mcX_G \cup
\mcX_G^{-1}$. So the above algorithm takes time polynomial in $N$ for words
from $\mcX_G \cup \mcX_G^{-1}$.
\end{proof}

We can now embed $G_{\mcL}$ into a finitely presented group:
\begin{proposition}\label{prop:HGroup} 
    There is a finitely presented group $H_\mcL = \ang{\mcX_H:\mcR_H}$, 
    a polynomial $d_{\mcL}$, and an embedding
    $\iota : \F_{\mcX_G} \incl \F_{\mcX_H}$ such that
    \begin{enumerate}[(a)]
        \item $z^2 \in \mcR_H$ for all $z \in \mcX_H$ (so all the generators of $H_{\mcL}$ are involutions), 
        \item $\iota(J)$, $\iota(X)$, and  $\iota(Z)$ are elements of $\mcX_H$, while $\iota(S)$, $\iota(T)$, and $\iota(W)$ have length
            two in $H_\mcL$, and
        \item if $w \in \F_{\mcX_G}$ is a word
    which is trivial in $G_{\mcL}$, then there 
    there exist $r_1,\cdots,r_k\in\mcR_H\cup\mcR_H^{-1}$, $z_1,\cdots,z_k\in
\F_{\mcX_H}$ with $k\leq d_\mcL(\abs{\iota(w)})$ and $\abs{z_i}\leq d_\mcL(\abs{\iota(w)})$
for all $1\leq i\leq k$ satisfying 
    \begin{align*}
        \iota(w)=z_1r_1z_1^{-1}z_2r_2z_2^{-1}\cdots z_kr_kz_k^{-1}
    \end{align*}
in $\F_{\mcX_H}$.
    \end{enumerate}
\end{proposition}

Item (c) implies that $\iota$ embeds $G_\mcL$ into $H_\mcL$. 

\begin{proof}
    By \Cref{prop:polyword} and \Cref{thm:BORS}, $G_{\mcL}$ can be embedded in
    a finitely presented group $H_0 = \langle \mcX_0 : \mcR_0 \rangle$ with
    polynomial isoperimetric function $d_{H_0}$. Let $\mcX_1 = \mcX_0 \setminus
    \{J,X,Z\}$, and $\mcX_H = \{J,X,Z\} \cup \{s_y, t_y : y \in \mcX_1\}$.
    Define $f : \F_{\mcX_0} \to \F_{\mcX_H}$ by $f(y) = y$ for $y \in
    \{J,X,Z\}$ and $f(y) = s_y t_y$ for $y\in \mcX_1$, so in other words $f(w)$
    is the word $w$ with $y$ replaced by $s_y t_y$ for all $y \in \mcX_1$. Finally,
    let $\mcR_H = \{z^2 : z \in \mcX_H\} \cup \{f(r) : r \in \mcR_0\}$. Then
    \begin{equation*}
        \ang{\mcX_H : \mcR_H} = \ang{H_0, s_y, t_y \text{ for } y \in \mcX_1 : s_y^2 = t_y^2 = 1 \text{ and } 
            y = s_y t_y \text{ for all } y \in \mcX_1},
    \end{equation*}
    and thus $H_0$ is a subgroup of $H := \ang{\mcX_H : \mcR_H}$ by repeatedly
    applying \Cref{lem:involution}. The inclusion $\iota : G_0 \incl H_0 \incl H$
    satisfies parts (a) and (b) of the proposition. Let $\ell$ be the length of
    the longest relation in $\mcR_0$. If $w \in \F_{\mcX_G}$ is
    trivial in $G_{\mcL}$, then $w$ is also trivial in $H_0$.  Hence by
    \Cref{lemma:isod}, there are $r_1,\ldots,r_k \in \mcR_0 \cup \mcR_0^{-1}$
    and $z_1,\ldots,z_k \in \F_{\mcX_0}$ such that 
    \begin{equation*}
        w = z_1 r_1 z_1^{-1} z_2 r_2 z_2^{-1} \cdots z_k r_k z_k^{-1},
    \end{equation*}    
    where $k \leq d_0(|w|)$ and $|z_i| \leq \ell k + \ell + |w|$ for all $1 \leq i \leq k$.
    Thus $\iota(w) = f(z_1) f(r_1) f(z_1)^{-1} \cdots f(z_k) f(r_k) f(z_k)^{-1}$.  
    Since $d_{H_0}$ is polynomial, $f(r_i) \in \mcR_H \cup \mcR_H^{-1}$ for all $1 \leq i \leq k$, and $|w|
    \leq |f(w)| \leq 2|w|$ for all $w \in \F_{\mcX_0}$, 
    part (c) is true with e.g. $d_\mcL(n) := 2 \ell d_{H_0}(n) + 2 \ell + 2n$.
\end{proof}

\subsection{Construction of $\mcA_\mcL(m)$}\label{ss:construction}
For this section, we fix a finitely presented group $H_{\mcL} = \ang{\mcX_H :
\mcR_H}$ and a polynomial $d_\mcL$ as in \Cref{prop:HGroup}, and identify
$G_{\mcL}$ with its image in $H_{\mcL}$ (so we don't mention the embedding
$\iota$ explicitly). In particular, if $w \in \F_{\mcX_G}$, then $|w|$ will
denote the length of $w$ in $H_{\mcL}$ rather than $G_\mcL $ unless otherwise
noted. The variables $J$, $X$, and $Z$ will continue to denote generators of
$H_{\mcL}$, while $S$, $T$ and $W$ will be length two words in $H_{\mcL}$. 
We also continue with the notation $X_{mi} = S^i W^m X W^{-m} S^{-i}$
and $Z_{mi} = T^i W^m Z W^{-m} T^{-i}$. Now that we are working with algebras,
we let $[A,B]$ denote the algebra commutator $AB-BA$ instead of the group
commutator. 

To define $\mcA_\mcL(m)$, consider the set of variables
\begin{equation*}
    \mcX_0:=\{U_1,U_2,\widetilde{X},\widetilde{Z},O_P,O_Q\},
\end{equation*}
and let
\begin{align*}
    U :=U_1U_2,\quad P:=\frac{\Id-O_P}{2}, \text{ and } Q :=\frac{\Id-O_Q}{2}
\end{align*}
in $\C^*\ang{\mcX_0}$. For every $m\in\N$, let $\mcA_{\mcL}(m)$ be the finitely
presented $*$-algebra generated by $\mcX_H \cup \mcX_0$, subject to the
relations
\begin{enumerate}[(R1)]
\setcounter{enumi}{-1}
    \item $x^* x = x x^* = x^2 = 1$ for all $x \in \mcX_H \cup \mcX_0$,
    \item $r=1$ for all $r \in \mcR_H$, i.e. all the relations for $H_{\mcL}$ hold, 
    \item $[U,X]=[U,Z]=[U,S]=[U,T]=[U,J]=[Q,X]=[Q,Z]=[Q,S]=[Q,T]=[Q,J]=0$, i.e. $U$ and $Q$ commute with $X$,$Z$,$S$,$T$, and $J$,
    \item $[\widetilde{X},Q] = \widetilde{X}Q-X_{m,0}Q = 0$ and $[\widetilde{Z},Q]=\widetilde{Z}Q-Z_{m,0}Q = 0$,
    \item $U\widetilde{X}U^*-S\widetilde{X}S^* = U\widetilde{Z}U^*-T\widetilde{Z}T^* = 0$,
    \item $[P,Q]=0$, i.e. $P$ commutes with $Q$, and
    \item $(P+\widetilde{X}P\widetilde{X}-UPU^*)(\Id+J\widetilde{X}\widetilde{Z}\widetilde{X}\widetilde{Z})=0$.
\end{enumerate}
Note that relations (R0) imply that all the generators are order-two unitaries,
so $\mcA_\mcL(m) = \C \Z_2^{*\mcX_H \cup \mcX_0} / \ang{\mcR_m}$, 
where $\mcR_m$ denotes the relations in (R1)-(R6). Since $O_P$ and $O_Q$
have order two, $P$ and $Q$ are projections in
$\mcA_{\mcL}(m)$. The point of the
generators $O_P$ and $O_Q$ is to add the projections $P$ and
$Q$ while making $\mcA_\mcL(m)$ a quotient of a free product
of $\Z_2$'s. Similarly, the point of $U_1$ and $U_2$ is to add the unitary
$U$ without putting any restriction on its order. For the intuition behind
the remaining relations, see the representation constructed in the proof of
\Cref{lemma:globaltrace} below.  Before proceeding to the lemma, we record some
facts about the presentation for later use:
\begin{remark}\label{rmk:mcr}\ %
    \begin{enumerate}
        \item If $\mcA = \C\Z_2^{*\mcX_H\cup\mcX_0}$, then $\norm{r}_{\mcA} \leq 6$ for all $r\in\mcR_m$.

        \item The number of relations in $\mcR_m$ is independent of $m$.

        \item The relations $r$ in (R3) satisfy $\norm{r}_{1,1}\leq 2m+3$, and the other relations in $\mcR_m$ do not depend on $m$.

    \end{enumerate}
\end{remark}

\begin{lemma}\label{lemma:globaltrace}
If $m\in\mcL$ then there is a tracial state $\tau$ on $\mcA_\mcL(m)$ such that $\tau(PQ)>0$.

\end{lemma}

\begin{proof}
Suppose $m\in\mcL$ and let $n:=h(m)<\infty$, so the Turing machine $M$ halts
at the $n$-th step on the input $m$. Let $\tau_0$ be the canonical trace on $\C
H_{\mcL}$, and let $(\pi:\C H_\mcL \arr \mcB(\mcH_0),v)$ be the GNS
representation of $\tau_0$. Let $\mcH_1$ be the Hilbert space
$\ell^2(\Z_{n+1})$ with the canonical basis $\{ e_i :i\in\Z_{n+1} \}$, let
$E_{i,i}$ be the rank-one projection onto $e_i$ for every $i\in\Z_{n+1}$, and
let $\zeta:e_i \mapsto e_{i+1}$ be the left cyclic shift operator on
$\ell^2(\Z_{n+1})$.  Let $\widetilde{\mcH}=\mcH_0\otimes\mcH_1$, and
define a $*$-representation
$\widetilde{\pi}:\C^*\ang{\mcX_H\cup\mcX_0}\arr\mcB(\widetilde{\mcH}\oplus\widetilde{\mcH})$
by 
\begin{align*}
    & \widetilde{\pi}(x)=
    \begin{pmatrix}
    \pi(x)\otimes Id_{\mcH_1} & \quad\\
    \quad & \pi(x)\otimes Id_{\mcH_1}
    \end{pmatrix} 
    \text{ for all } x\in\mcX_H,\\   
   & \widetilde{\pi}(U_1)=
    \begin{pmatrix}
    \quad & Id_{\mcH_0}\otimes \zeta\\
    Id_{\mcH_0}\otimes \zeta^{-1}
    \end{pmatrix}, \quad
    \widetilde{\pi}(U_2)=
    \begin{pmatrix}
    \quad & Id_{\widetilde{\mcH}}\\
    Id_{\widetilde{\mcH}} 
    \end{pmatrix},\\
    &\widetilde{\pi}(\widetilde{X})=
    \begin{pmatrix}
    \sum\limits_{i=0}^n \pi(X_{mi})\otimes E_{-i,-i} & \quad\\
    \quad & \sum\limits_{i=0}^n \pi(X_{mi})\otimes E_{i,i} 
    \end{pmatrix},\\
    &\widetilde{\pi}(\widetilde{Z})=
    \begin{pmatrix}
    \sum\limits_{i=0}^n \pi(Z_{mi})\otimes E_{-i,-i} &\quad\\
    \quad & \sum\limits_{i=0}^n \pi(Z_{mi})\otimes E_{i,i}
    \end{pmatrix},\\ 
    &\widetilde{\pi}(Q)=
    \begin{pmatrix}
    Id_{\mcH_0}\otimes E_{0,0} & \quad\\
    \quad & Id_{\mcH_0}\otimes E_{0,0}  
    \end{pmatrix}, \text{ and }\\
    &\widetilde{\pi}(P)=
    \begin{pmatrix}
    \sum\limits_{i=0}^n \pi(P_i)\otimes E_{-i,-i} & \quad\\
    \quad & \sum\limits_{i=0}^n \pi(P_i)\otimes E_{i,i}
    \end{pmatrix},
\end{align*}
   where in the last equation $P_i:=(\frac{\Id-J}{2})(\frac{\Id-Z_{m,i}}{2})(\frac{\Id-Z_{m,i+1}}{2})\cdots(\frac{\Id-Z_{m,n-1}}{2})$ for all $0\leq i\leq n-1$, and $P_n:=\frac{\Id-J}{2}$. Note that 
   \begin{align*}
      & \widetilde{\pi}(U)=\widetilde{\pi}(U_1)\widetilde{\pi}(U_2) = 
    \begin{pmatrix}
     Id_{\mcH_0}\otimes \zeta & \quad\\
   \quad & Id_{\mcH_0}\otimes \zeta^{-1}
    \end{pmatrix}.
\end{align*}
   
It is not hard to check that $\widetilde{\pi}(r)=0$ for all relations $r$ in
(R0)-(R6).  Indeed, it is clear that $\widetilde{\pi}(x)$ is a unitary squaring
to $\Id$ for $x\in\mcX_H\cup\{U_1,U_2,\widetilde{X},\widetilde{Z}\}$, and
$\widetilde{\pi}(Q)$ is a projection. Since $Z_{m,0},\cdots,Z_{m,n-1}$ and $J$
commute in $\C H_\mcL$, $\widetilde{\pi}(P)$ is a projection as well. So
$\widetilde{\pi}(r)=0$ for all $r$ in (R0). We also have $\widetilde{\pi}(r)=0$
for $r$ in (R1), (R2), and (R5) by
inspection. For (R3), 
\begin{equation*}
    \widetilde{\pi}(\widetilde{X})\widetilde{\pi}(Q) = \widetilde{\pi}(Q)\widetilde{\pi}(\widetilde{X}) =  
    \widetilde{\pi}(X_{m,0}) \widetilde{\pi}(Q) = 
    \begin{pmatrix} \pi(X_{m,0})\otimes E_{0,0} & \quad\\
        \quad & \pi(X_{m,0})\otimes E_{0,0} 
    \end{pmatrix},
\end{equation*}
and similarly $  \widetilde{\pi}(\widetilde{Z})\widetilde{\pi}(Q) =  \widetilde{\pi}(Q)\widetilde{\pi}(\widetilde{Z}) = \widetilde{\pi}(Z_{m,0}) \widetilde{\pi}(Q)$. For (R4), recall from the presentation of $G_{\mcL}$ that $\pi(X_{m,0}) =\pi(X_{m,n+1}) $ and $ \pi(S)\pi(X_{m,i}) \pi(S^*) = \pi(X_{m,i+1})$ for all $i \in \Z$. Since $\zeta E_{i,i}\zeta^{-1}=E_{i+1,i+1}$ for all $i\in\Z_{n+1}$,
\begin{align*}
    \widetilde{\pi}(S) \widetilde{\pi}(\widetilde{X}) \widetilde{\pi}(S^*) &= \begin{pmatrix}
       \sum\limits^{n}_{i=0} \pi(X_{m,i+1})\otimes E_{-i,-i} & \quad\\
    \quad & \sum\limits^{n}_{i=0} \pi(X_{m,i+1})\otimes E_{i,i}
    \end{pmatrix}\\ &= \begin{pmatrix}
       \sum\limits^{n}_{i=0} \pi(X_{m,i})\otimes E_{-(i-1),-(i-1)} & \quad\\
    \quad & \sum\limits^{n}_{i=0} \pi(X_{m,i})\otimes E_{i-1,i-1}
    \end{pmatrix}\\ &= \widetilde{\pi}(U) \widetilde{\pi}(\widetilde{X}) \widetilde{\pi}(U^*).
\end{align*}
Similarly $\widetilde{\pi}(T) \widetilde{\pi}(\widetilde{Z})
\widetilde{\pi}(T^*) = \widetilde{\pi}(U) \widetilde{\pi}(\widetilde{Z})
\widetilde{\pi}(U^*)$. This just leaves (R6). For this case,
\begin{align*}
 &\wtd{\pi}\big( (P+\widetilde{X}P\widetilde{X}-UPU^*)(\Id+J\widetilde{X}\widetilde{Z}\widetilde{X}\widetilde{Z})\big)=\begin{pmatrix}
    A_1 & \quad\\
    \quad & A_2
\end{pmatrix} \text{ where}\\   &A_1=\sum_{i=0}^n\pi\big((P_i+X_{mi}P_{i}X_{mi}-P_{i+1})(1+JX_{mi}Z_{mi}X_{mi}Z_{mi})\big)\otimes E_{-i,-i},\\
&A_2=\sum_{i=0}^n\pi\big((P_i+X_{mi}P_{i}X_{mi}-P_{i+1})(1+JX_{mi}Z_{mi}X_{mi}Z_{mi})\big)\otimes E_{i,i}, 
\end{align*}
and $P_{n+1}:=P_0$.  Recall from the presentation of $G_\mcL$ that $J$ is a
unitary of order 2 and is central. Hence in $\C H_\mcL$, $\tfrac{1- J}{2}$ and
$\tfrac{1+J}{2}$ are projections, and $(\tfrac{1+J}{2})P_i=0$ for $0\leq i\leq
n$. Also for all $0\leq i\leq n-1$, $X_{mi}Z_{mi}X_{mi}Z_{mi}=J$, $X_{mi}$ commutes
with $P_{i+1}$, and $P_i=(\tfrac{1-Z_{mi}}{2})P_{i+1}$. Hence
$1+JX_{mi}Z_{mi}X_{mi}Z_{mi}=2$, and
\begin{align*}
P_i+X_{mi}P_{i}X_{mi}-P_{i+1}
=(\tfrac{1-Z_{mi}}{2})P_{i+1}+(\tfrac{1-JZ_{mi}}{2})P_{i+1}-P_{i+1}=-Z_{mi}(\tfrac{1+J}{2})P_{i+1}=0.  
\end{align*}
For $i=n$, $X_{mn}Z_{mn}X_{mn}Z_{mn}=1$, so 
\begin{align*}
    (P_n+X_{mn}P_nX_{mn}-P_0)(1+ JX_{mn}Z_{mn}X_{mn}Z_{mn})=(2-P_0)\tfrac{(1-J)(1+J)}{2}=0.
\end{align*}
 It follows that $\pi\big((P_i+X_{mi}P_{i}X_{mi}-P_{i+1})(1+JX_{mi}Z_{mi}X_{mi}Z_{mi})\big)=0$ for all $0\leq i\leq n$, and hence $A_1=A_2=0$. We conclude that $\wtd{\pi}(r)=0$ for all  $r$ in (R0)-(R6). Hence $\wtd{\pi}$ induces a representation of $\mcA_{\mcL}(m)$ on $\wtd{\mcH}\oplus\wtd{\mcH}$.

The image of $\wtd{\pi}$ is contained in $M_2\big(\pi(\C H_\mcL)\otimes
\mcB(\mcH_1) \big)$ This algebra is isomorphic to $ \pi(\C
H_\mcL) \otimes M_{2(n+1)}(\C)$, which has a tracial state
$\tau:=\tau_0\otimes \tau'$ where $\tau'$ is the unique tracial state on
$M_{2(n+1)}(\C)$. Finally, 
\begin{align*}
     \tau\big(\wtd{\pi}(PQ)\big)&=\tau\begin{pmatrix}
         \pi(P_0)\otimes E_{0,0} & \quad\\
         \quad & \pi(P_0)\otimes E_{0,0}
     \end{pmatrix}
     =\tfrac{1}{(n+1)}\tau_0(P_0)=\tfrac{1}{2^{n+1}(n+1) } >0,
 \end{align*}
 so the lemma follows.
\end{proof}

We now look at what happens when $m \not\in \mcL$. For all $n \geq 0$, let 
\begin{equation*}
    \widetilde{P}_n:=QU^nPU^{-n}Q \text{ and } \widetilde{X}_n:=U^n\widetilde{X}U^{-n}
\end{equation*}
in $\C^*\ang{\mcX_H \cup \mcX_0}$. Note that $\wtd{X}_n^2 = 1$ in $\C \Z_2^{*\mcX_H \cup \mcX_0}$. 

\begin{proposition}\label{prop:sizeofRdecomp}
    There are constants $C$ and $k$ such that for any $m \in \mbN$ and $0 \leq
    n < h(m)$, $ \wtd{P}_n + \wtd{X}_n \wtd{P}_n \wtd{X}_n-\wtd{P}_{n+1}$
    is trivial in $\mcA_{\mcL}(m) = \C \Z_2^{* \mcX_H \cup \mcX_0} /
    \ang{\mcR_m}$, and has an $\mcR_m$-decomposition in $\C \Z_2^{*\mcX_H \cup
    \mcX_0}$ of size $\leq C \big((n+1)m\big)^k$. 
\end{proposition}
Although we won't need this, the constants $C$ and $k$ can be computed from a
description of the Turing machine $M_{\mcL}$ that recognizes $\mcL$.
\begin{proof}
For the proof, say that $\alpha(n)=O(\beta(n))$ for functions $\alpha,\beta$ if there is a constant $C>0$ such that $\alpha(n)\leq C\beta(n)+C$ for all $n\geq 0$.
    Recall from \Cref{rmk:mcr} that $\norm{r}_{\mcA} \leq 6$ for all $r \in
    \mcR_m$ when $\mcA = \C \Z_2^{*\mcX_H \cup \mcX_0}$. Let
    \begin{equation*}
        f =  QU^n(P + \wtd X P \wtd X - U P U^*)U^{-n}\text{ and }g = U^n(1 + J \wtd X \wtd Z \wtd X \wtd Z)U^{-n} Q.
    \end{equation*}
    By relation (R6), $fg$ is trivial in $\mcA_{\mcL}(m)$ and has an $\mcR_m$-decomposition in $\mcA$
    of size $6n+4$. To prove the proposition, we show that 
    \begin{equation*}
        fg = 2 \left(\wtd P_n + \wtd X_n \wtd P_n \wtd X_n - \wtd P_{n+1}\right)
    \end{equation*}
    in $\mcA_{\mcL}(m)$, while keeping track of the size using
    \Cref{lemma:sizecalc}. For the first step, since $U$ and $S$ commute and $U
    \wtd X U^* = S \wtd X S^*$, we see that $\wtd X_n = S^n \wtd X S^{-n}$ in
    $\mcA_{\mcL}(m)$. Similarly $U^n \wtd Z U^{-n} = T^n \wtd Z T^{-n}$ and $U$
    commutes with $J$, so
    \begin{equation*}
        g = (1 + J S^n \wtd X S^{-n} \cdot T^n \wtd Z T^{-n} \cdot S^n \wtd X S^{-n} \cdot T^n \wtd Z T^{-n}) Q.
    \end{equation*}
    Replacing $U^n \wtd X U^{-n}$ and $U^{n} \wtd Z U^{-n}$ with $S^n \wtd X S^{-n}$
    and $T^n \wtd Z T^{-n}$ requires $O(n^2)$ applications of the defining relations,
    and the intermediate polynomials in this calculation have degree $O(n)$, so 
    by Lemma \ref{lemma:sizecalc}, 
    \begin{equation*}
        g - (1 + J S^n \wtd X S^{-n} \cdot T^n \wtd Z T^{-n} \cdot S^n \wtd X S^{-n} \cdot T^n \wtd Z T^{-n}) Q
    \end{equation*}
    has an $\mcR_m$-decomposition of size $O(n^3)$. 

    Next, $\wtd X_n Q = Q \wtd X_n = S^n \wtd X Q S^{-n} = Q X_{mn}$ and
    $U^n \wtd Z U^{-n} Q = Q Z_{mn}$ for all $n \geq 0$ by relations (R2) and (R3), 
    so
    \begin{equation*}
        (1 + J S^n \wtd X S^{-n} \cdot T^n \wtd Z T^{-n} \cdot S^n \wtd X S^{-n} \cdot T^n \wtd Z T^{-n}) Q
        = Q (1 + J X_{mn} Z_{mn} X_{mn} Z_{mn}).
    \end{equation*}
    This identity requires using the relations $O(n)$ times, and $X_{mn}$ and $Z_{mn}$ have degree $O(m+n)$, 
    so by \Cref{lemma:sizecalc} again,
    \begin{equation*}
        (1 + J S^n \wtd X S^{-n} \cdot T^n \wtd Z T^{-n} \cdot S^n \wtd X S^{-n} \cdot T^n \wtd Z T^{-n}) Q
        - Q (1 + J X_{mn} Z_{mn} X_{mn} Z_{mn})
    \end{equation*}
    has an $\mcR_m$-decomposition of size $O(n^2 + nm)$. 
    
    Now suppose that $0 \leq n < h(m)$. From the presentation of $G_{\mcL}$, we see
    that $X_{mn} Z_{mn} X_{mn} Z_{mn} = J$, so $Q (1 + J X_{mn} Z_{mn} X_{mn}
    Z_{mn}) = Q (1 + J^2) = 2 Q$. The word $J X_{mn} Z_{mn} X_{mn} Z_{mn}$ has
    length $16m + 16n + 5$, so by \Cref{prop:HGroup}, part (c) there are
    $r_1,\ldots,r_k \in \mcR_H \cup \mcR_H^{-1}$ and $z_1,\ldots,z_k \in
    \F_{\mcX_H}$ with $J X_{mn} Z_{mn} X_{mn} Z_{mn} = z_1 r_1 z_1^{-1} \cdots
    z_{k} r_k z_k^{-1}$ in $\F_{\mcX_H}$, $k \leq d_{\mcL}(16m + 16n + 5)$, and
    $|z_i|\leq d_{\mcL}(16m+16n+5)$ for all $1 \leq i \leq k$. Replacing $r_i$
    with $1$ from right to left, we see that $Q (1 + J X_{mn} Z_{mn} X_{mn}
    Z_{mn}) - 2Q$ has an $\mcR_m$-decomposition of size $ O(d_\mcL(16m+16n+5)^2)
    = O(d_\mcL(m)^2 + d_\mcL(n)^2)$. From \Cref{lemma:sizecalc}, part (c), we get
    that $fg - 2fQ$ has an $\mcR_m$-decomposition of size
    $O(n^3 + n^2 + nm + d_{\mcL}(m)^2 + d_\mcL(n)^2)$. 

    We now look at $fQ$. Since $Q$ commutes with $\wtd X_n$, 
    \begin{align*}
        fQ & = Q U^n P U^{-n} Q + Q U^n \wtd X P \wtd X U^{-n} Q - Q U^{n+1} P U^{-(n+1)} Q \\
          & = \wtd P_{n} + Q \wtd X_{n} \cdot U^n P U^{-n} \cdot \wtd X_n  Q - \wtd P_{n+1} \\
          & = \wtd P_{n} + \wtd X_{n} \cdot \wtd P_{n} \cdot \wtd X_n  - \wtd P_{n+1}.
    \end{align*}
    This identity requires replacing $\wtd X_n$ with $S^n \wtd X S^{-n}$ again, so
    \begin{equation*}
        fQ - (\wtd P_{n} + \wtd X_{n} \cdot \wtd P_{n} \cdot \wtd X_n  - \wtd P_{n+1})
    \end{equation*}
    has an $\mcR_m$-decomposition of size $ O(n^3)$. Ultimately we conclude that
    $\wtd P_{n} + \wtd X_{n} \cdot \wtd P_{n} \cdot \wtd X_n  - \wtd P_{n+1}$ has an
    $\mcR_m$-decomposition of size $ O(n^3 + n^2 + nm + d_{\mcL}(m)^2 + d_\mcL(n)^2)$.
\end{proof}

\begin{proof}[Proof of \Cref{thm:main}]

    Let $\mcX = \mcX_H \cup \mcX_0$, and let $\wtd{P}_n$, $\wtd{X}_n$, and
    $\mcR_m$ be as above. The relations in (R1) are independent of $m$, and the
    remaining relations are spelled out explicitly in (R2)-(R6), so we see that
    the mapping $m \mapsto \mcR_m$ is computable. Parts (a) and (d) of the theorem hold
    by construction, and part (b) follows from \Cref{prop:sizeofRdecomp}.
    Finally, since $[P,Q]=0$ in $\mcA_{\mcL}(m)$, \Cref{lemma:globaltrace} implies
    that there is a tracial state $\tau$ on $\mcA_{\mcL}(m)$ with $\tau(\wtd{P}_0^2) 
    = \tau(PQ) > 0$.
\end{proof}

\section{Proof of Theorem \ref{thm:maintracial}}

In this section, we prove the following theorem.
\begin{theorem}\label{thm:actualmaintracial}
    There is an integer $n_0$, such that for all $N \geq n_0$, the promise problem
    \begin{quote}
        Given $\alpha \in \Q^*\ang{x_1,\ldots,x_N}$, determine if $\alpha$ is
        trace-positive in $\C^*\ang{x_1,\ldots,x_N}$, or if $\pi(\alpha)$ is
        not trace-positive in $\C \Z_2^{*N}$, promised that one of the two is 
        the case. 
    \end{quote}
    is coRE-hard, where $\pi : \Q^*\ang{x_1,\ldots,x_N} \to \Q \Z_2^{*N}$ is the
    homomorphism sending $x_i$ to the generator of the $i$th factor of $\Z_2^{*N}$.
\end{theorem}
Suppose that $N \geq n_0$, and that the homomorphism $\pi$ in
\Cref{thm:actualmaintracial} factors through a $*$-algebra $\mcA$, meaning that
there are homomorphisms $\pi_1 : \Q^*\ang{x_1,\ldots,x_N} \to \mcA$ and $\pi_2
: \mcA \to \Q \Z_2^{*N}$ such that $\pi = \pi_2 \circ \pi_1$. Suppose also that
we have some way of representing the elements of $\mcA$ as strings such that
the homomorphism $\pi_1$ is computable.  Then \Cref{thm:actualmaintracial} has
the following immediate corollary:
\begin{cor}\label{cor:actualmaintracial}
    It is coRE-hard to determine whether $\alpha \in \mcA$ is trace-positive in
    $\mcA_{\C}$. 
\end{cor}
\begin{proof}
    Suppose $\alpha \in \Q^*\ang{x_1,\ldots,x_N}$. If $\alpha$ is trace-positive
    in $\C^*\ang{x_1,\ldots,x_N}$, then $\pi_1(\alpha)$ is trace-positive in
    $\mcA_{\C}$. If $\pi(\alpha)$ is not trace-positive in $\C \Z_2^{*N}$, then
    there is a tracial state $\tau$ on $\C \Z_2^{*N}$ such that $(\tau \circ \pi_2)
    (\pi_1(\alpha)) = \tau(\pi(\alpha)) < 0$, so $\pi_1(\alpha)$ is not trace-positive
    in $\mcA_{\C}$. Since $\pi_1$ is computable, the promise problem in \Cref{thm:actualmaintracial} reduces
    to the problem of determining if $\alpha \in \mcA$ is trace-positive. 
\end{proof}
\Cref{thm:maintracial} is a special case of \Cref{cor:actualmaintracial}
with $\mcA = \Q^*\{x_1,\ldots,x_N\}$, $\Q \F_N$, or $\Q \Z_M^{*N}$ (for every $M\geq M'$, there is a $*$-homomorphism $\Q\Z_{M}^{*N}\arr\Q\Z_{M'}^{*N}$ sending  the $i$th generator of $\Q\Z_{M}^{*N}$ to the $i$th generator of $\Q\Z_{M'}^{*N}$).
The corollary also applies to other algebras such as the algebra of
contractions and the algebra of self-adjoint contractions, as well as free
products of any of the algebras mentioned.

For the proof of \Cref{thm:actualmaintracial}, suppose $\mcL \subseteq \N$ is
an RE set, and let $\mcX$, $\wtd{P}_n$ and $\wtd{X}_n$, $n \geq 0$, and
$\mcA_{\mcL}(m) = \C^*\ang{\C \Z_2^{*\mcX} : \mcR_m}$, $m \geq 1$, be as in
\Cref{thm:main}. Before proceeding to the proof, we need the following
propositions.

\begin{proposition}\label{prop:nonhalting}
    There are positive integers $\Lambda,k$ such that if $m\notin \mcL$ and $\tau$
    is a tracial $(\epsilon,\mcR_m)$-state on $\C\Z_2^{*\mcX}$ for some
    $\epsilon\geq 0$, then $\norm{\wtd{P}_0}_{\tau}\leq \Lambda m^{k}\sqrt{\eps}$.
\end{proposition}

\begin{proof}
    Let $C$ and $k$ be the constants from \Cref{thm:main}, part (b). Suppose
    $m\notin \mcL$, $\eps\geq 0$, and $\tau$ is a tracial $(\epsilon,\mcR_m)$-state
    on $\C\Z_2^{*\mcX}$. By \Cref{thm:main}, part (a), $\wtd{X}_n$ is a unitary 
    involution in $\C\Z_2^{*\mcX}$. Since $\tau$ is tracial,
    $\norm{\wtd{P}_n}_\tau = \norm{\widetilde{X}_n \widetilde{P}_n\wtd{X}_n}_\tau$
    for all $n$. Also $\wtd{P}_n$ and $\widetilde{X}_n \widetilde{P}_n\wtd{X}_n$
    are hermitian squares in $\C\Z_2^{*\mcX}$, so
    $\tau(\widetilde{P}_n \widetilde{X}_n \widetilde{P}_n \widetilde{X}_n) =
    \tau(\widetilde{X}_n\widetilde{P}_n\widetilde{X}_n\widetilde{P}_n)\geq 0$. Hence 
    \begin{align*}
        \norm{\widetilde{P}_n}_{\tau}^2=\tau(\widetilde{P}_n^2)&=\tfrac{1}{2}\Big(\tau(\widetilde{P}_n^2)+\tau\big(( \widetilde{X}_n\widetilde{P}_n\widetilde{X}_n )^2  \big)  \Big)\\
        &\leq\tfrac{1}{2}\Big(\tau(\widetilde{P}_n^2)+\tau\big(( \widetilde{X}_n\widetilde{P}_n\widetilde{X}_n )^2  \big)+\tau(\widetilde{P}_n\widetilde{X}_n\widetilde{P}_n\widetilde{X}_n)+\tau(\widetilde{X}_n\widetilde{P}_n\widetilde{X}_n\widetilde{P}_n)  \Big)\\
        &=\tfrac{1}{2}\tau\big((\widetilde{P}_n+ \widetilde{X}_n\widetilde{P}_n\widetilde{X}_n )^2  \big)=\tfrac{1}{2}\norm{\widetilde{P}_n+ \widetilde{X}_n\widetilde{P}_n\widetilde{X}_n}_{\tau}^2.
    \end{align*} 
    By \Cref{prop:tracerbound} and \Cref{thm:main}, part (b), $\norm{\widetilde{P}_n + \widetilde{X}_n \widetilde{P}_n
    \widetilde{X}_n-\widetilde{P}_{n+1}}_{\tau}\leq C((n+1)m)^k\sqrt{\eps}$ for all
    $n\geq 0$. It follows that
    \begin{align*}
        \norm{\widetilde{P}_n}_{\tau}&\leq\tfrac{1}{\sqrt{2}}\norm{\widetilde{P}_n+ \widetilde{X}_n\widetilde{P}_n\widetilde{X}_n}_{\tau}\\
        &\leq \tfrac{1}{\sqrt{2}}(\norm{\widetilde{P}_{n+1}}_{\tau}+\norm{\widetilde{P}_n + \widetilde{X}_n \widetilde{P}_n \widetilde{X}_n-\widetilde{P}_{n+1}}_{\tau}  )\\
        &\leq \tfrac{1}{\sqrt{2}}\norm{\widetilde{P}_{n+1}}_{\tau}+\tfrac{1}{\sqrt{2}}C((n+1)m)^k\sqrt{\epsilon}.
    \end{align*}
    By induction, we obtain
    \begin{align*}
        \norm{\wtd{P}_0}_\tau\leq (\tfrac{1}{\sqrt{2}})^{t}\norm{\wtd{P}_{t}}_\tau+C\sum_{n=1}^{t}\frac{ n^k}{2^{n/2}} m^k\sqrt{\eps}
    \end{align*}
    for all $t\geq 1$. Since $\wtd{P}_t$ is a contraction in
    $\C\Z_2^{*\mcX}$ by \Cref{thm:main}, part (a), $\norm{\wtd{P}_t}_\tau\leq 1$ for all $t\geq 1$.
    The series $\sum_{n=1}^{\infty}\frac{ n^k}{2^{n/2}}$ converges, so taking
    $t\arr + \infty$, we conclude that 
    \begin{align*}
        \norm{\wtd{P}_0}_\tau\leq \Lambda m^k\sqrt{\eps}
    \end{align*}
    for any integer $\Lambda\geq C \sum_{n=1}^{\infty}\frac{ n^k}{2^{n/2}}$.
\end{proof}

\begin{proposition}\label{prop:actaultracial}
Let $\mcW$ be the mapping from \Cref{lemma:tracialfree}, and let 
\begin{align*}
\mcW_m:=\left(\bigcup_{r\in\mcR_m\cup\mcR_m^*}\mcW(r)\right)\cup\mcW(\wtd{P}_0)\cup\mcR_m
\end{align*} 
for $m\in\N$.
There are positive integers $\wtd{\Lambda},k'$ such that if $m\notin\mcL$ and $\tau$ is a tracial $(\eps,\mcW_m)$-state on $\C^*\ang{\mcX}$, then $\norm{\wtd{P}_0}_\tau\leq \wtd{\Lambda} m^{k'}\sqrt{\eps}$.  
\end{proposition}

\begin{proof}
Let $\Lambda,k$ be integers from \Cref{prop:nonhalting}, and let $D$ be the integer from \Cref{thm:main}, part (d). Suppose $m\notin\mcL$, $\eps\geq 0$, and $\tau$ is a tracial $(\eps,\mcW_m)$-state on $\C^*\ang{\mcX}$. By \Cref{lemma:tracialfree}, there is a tracial state $\widehat{\tau}$ on
    $\C\Z_2^{*\mcX}$ such that 
    \begin{align*}
    \big\vert
    \norm{r}_{\widehat{\tau}}-\norm{r}_\tau\big\vert\leq
    2\norm{r}_{1,1}\sqrt{\eps}\leq 2Dm\sqrt{\eps}
    \end{align*}
    for all
    $r\in\mcR_m\cup\mcR_m^*$. Since $\mcR_m\subseteq \mcW_m$, $\tau$ is an $(\eps,\mcR_m)$-state. For every $r\in\mcR_m\cup \mcR_m^*$, $\norm{r}_{\tau}\leq \sqrt{\eps}$, and hence $\norm{r}_{\widehat{\tau}}\leq (2Dm+1)\sqrt{\eps}$.
    We conclude that $\widehat{\tau}$ is a tracial
    $\big((2Dm+1)^2\eps,\mcR_m\big)$-state on $\C\Z_2^{*\mcX}$. By
    \Cref{prop:nonhalting}, $\norm{\wtd{P}_0}_{\widehat{\tau}}\leq\Lambda
    m^{k}\sqrt{(2Dm+1)^2\eps}=(2Dm+1)\Lambda
    m^{k}\sqrt{\eps}$. Again by \Cref{lemma:tracialfree}, since $\tau(s^*s)\leq
    \eps$ for all $s\in\mcW(\wtd{P}_0)$, we have 
    \begin{align*}
    \big\vert
    \norm{\wtd{P}_0}_{\widehat{\tau}}-\norm{\wtd{P}_0}_\tau\big\vert\leq
    2\norm{\wtd{P}_0}_{1,1}\sqrt{\eps}=3\sqrt{\eps}.  
    \end{align*}
It follows that
\begin{align*}
    \norm{\wtd{P}_0}_{\tau}\leq \left((2Dm+1)\Lambda m^{k}+3 \right)\sqrt{\eps}\leq
   6D\Lambda m^{k+1}\sqrt{\eps}.
\end{align*}
Taking $\wtd{\Lambda}:=6D\Lambda$ and $k':=k+1$ completes the proof. 
\end{proof}

\begin{proof}[Proof of \Cref{thm:actualmaintracial} ] 
    Let $n_0:=\abs{\mcX}$, and suppose we are given an integer $N\geq n_0$. 
    Pick an ordering on $\mcX$ so we can identify $\Q^*\ang{\mcX}$ with the
    subalgebra $\Q^*\ang{x_1,\ldots,x_{n_0}}$ of
    $\Q^*\ang{x_1,\ldots,x_N}$. We think of $\C^*\ang{\mcX}$ and $\C\Z_2^{*\mcX}$
    as subalgebras of  $ \C^*\ang{x_1,\ldots,x_N}$ and $\C\Z_2^{*N}$ similarly. Let
    $\wtd{\Lambda}$, $k'$, and $\mcW_m$ be as in \Cref{prop:actaultracial}.  
    For any $m\in\N$, define a $*$-polynomial 
    \begin{align}
         \alpha(m)=\sum_{r\in\mcW_m\cup\mcW_m^*}r^*r-\frac{1}{\wtd{\Lambda}^2m^{2k'}} \wtd{P}_0^*\wtd{P}_0 \in \Q^*\ang{\mcX}\subseteq\Q^*\ang{x_1,\ldots,x_N}.
    \end{align}
    Then $\alpha(m)$ is self-adjoint and the map $\alpha:\N \arr
    \Q^*\ang{x_1,\ldots,x_N}$ is computable by \Cref{thm:main,lemma:tracialfree}.

    Let $\pi : \Q^*\ang{x_1,\ldots,x_N} \to \Q \Z_2^{*N}$ be the homomorphism
    sending $x_i$ to the $i$th generator of $\Z_2^{*N}$. Also let $q = q_2 \circ q_1$, 
    where $q_1 : \Q \Z_2^{*N} \to \Q \Z_2^{*\mcX}$ is the homomorphism sending
    $x_i \mapsto x_i$ for $1 \leq i \leq n_0$ and $x_i \mapsto 1$ for $i >
    n_0$, and $q_2 : \Q \Z_2^{*\mcX} \to \mcA_{\mcL}(m)$ is the quotient
    map. By \Cref{thm:main}, if $m \in \mcL$ then there is a tracial state
    $\tau$ on $\mcA_{\mcL}(m)$ such that $\tau(\wtd{P}_0^2) > 0$. By
    \cref{lemma:tracialfree}, $\pi(t)=0$ for all $t \in
    \big(\bigcup_{r\in\mcR_m\cup\mcR_m^*}\mcW(r)\big)\cup\mcW(\wtd{P}_0)$.
    Since $q(r)=0$ for all $r \in \mcR_m$, $q(\pi(r^* r)) = 0$ for all
    $r \in \mcW_m \cup \mcW_m^*$, and hence $(\tau \circ q)(\pi(\alpha(m))) =
    -\frac{1}{\wtd{\Lambda}^2 m^{2k'}}\tau(\wtd{P}_0^2)<0$. We conclude that
    $\pi(\alpha(m))$ is not trace-positive. 

    Now suppose $m\notin\mcL$, and that $\tau$ is a tracial state on $\C^*\ang{x_1, \ldots, x_N}$. 
    Let $\eps := \sum_{r\in\mcW_m\cup\mcW_m^*}\tau(r^*r)$. Then $\tau$
    is a tracial $(\eps,\mcW_m)$-state, so by \Cref{prop:actaultracial},
    $\tau(\wtd{P}^{*}_{0}\wtd{P}_{0})=\norm{\wtd{P}_0}^2_{\tau}\leq \wtd{\Lambda}^2m^{2k'}\eps$. Hence
    $\tau(\alpha(m))\geq 0$. Thus the complement $\overline{\mcL}$ of $\mcL$
    reduces to the promise problem in the theorem statement. Taking $\mcL$ to
    be RE-hard, we see that the promise problem is coRE-hard. 
\end{proof}

\section{Proof of Theorem \ref{thm:mainbipartite}}\label{sec:actualmain}

The proof of Theorem \ref{thm:mainbipartite} follows the same approach as the proof of
Theorem \ref{thm:maintracial}. We start with a promise problem version of the theorem.
\begin{theorem}\label{thm:actualmain}
    Suppose either $N \geq 2$ and $\mcC=\Q\Z_3^{*N}$, or $N \geq 3$ and $\mcC=\Q\Z_2^{*N}$. Let $\mcP = \Q^*\ang{x_1,\ldots,x_N}$, and let $\pi : \mcP \to \mcC$ be the homomorphism
    sending $x_i$ to the $i$th generator of $\mcC$. Then the promise problem
    \begin{quote}
        Given $\alpha \in \mcP \otimes_{\Q} \mcP$, determine if $\alpha$ is
        positive in $\mcP_{\C} \otimes_{\C} \mcP_{\C}$, or if $(\pi \otimes
        \pi)(\alpha)$ is not positive in $\mcC_{\C} \otimes_{\C} \mcC_{\C}$,
        promised that one of the two is the case. 
    \end{quote}
    is coRE-hard. 
\end{theorem}
To get \Cref{thm:mainbipartite} from \Cref{thm:actualmain}, suppose that $N$,
$\mcC$, $\mcP$, and $\pi$ are as in the theorem,  and that
$\pi \otimes \pi$ factors through a $*$-algebra $\mcA$, meaning that there are
homomorphisms $\pi_1 : \mcP \otimes \mcP \to \mcA$ and $\pi_2 : \mcA \to \mcC
\otimes \mcC$ such that $\pi \otimes \pi = \pi_2 \circ \pi_1$. Suppose also
that the first homomorphism $\pi_1$ in this factorization is computable. Then:

\begin{cor}\label{cor:actualmain}
    It is coRE-hard to determine whether $\alpha \in \mcA$ is positive in
    $\mcA_{\C}$. 
\end{cor}
The proof is the same as the proof of \Cref{cor:actualmaintracial}, with the
word ``trace'' deleted. \Cref{thm:mainbipartite} follows immediately from
\Cref{cor:actualmain} by
taking $\mcA$ to be one of $\Q^*\ang{x_1,\ldots,x_N} \otimes 
\Q^*\ang{x_1,\ldots,x_N}$ or $\Q \mcF_N \otimes \Q \mcF_N$ with $N \geq 2$,
 or $\Q \Z_M^{*N} \otimes \Q \Z_M^{*N}$ with $N \geq 3,M\geq 2$ or $N \geq 2,M\geq 3$. As with
\Cref{cor:actualmaintracial}, \Cref{cor:actualmain} also applies to other
algebras such as the algebra of contractions and the algebra of self-adjoint
contractions, as well as free products of any of the algebras mentioned.

For the proof of \Cref{thm:actualmain}, we use the same setup as in the
previous section: we suppose $\mcL\subseteq\N$ is an RE set, and let
$\mcX,\wtd{P}_n,\wtd{X}_n$, $n\geq 0$, and
$\mcA_\mcL(m)=\C^*\ang{\C\Z_2^{*\mcX}:\mcR_m}$, $m\in\N$, be as in
\Cref{thm:main}. In addition, let $\iota: \C\Z_2^{*\mcX}\arr
\C\Z_2^{*\mcX}\otimes \C\Z_2^{*\mcX}:a\mapsto a\otimes\Id$ be the left
inclusion.

\begin{proposition}\label{prop:productnonhalting}
    There are positive integers $\Gamma,k$ such that if $m\notin\mcL$, $\varphi$ is an $(\eps,\mcX)$-synchronous state on $\C\Z_2^{*\mcX}\otimes \C\Z_2^{*\mcX}$, and $\tau:=\varphi\circ\iota$ is an $(\eps,\mcR_m)$-state on $\C\Z_2^{*\mcX}$ for some $\eps\geq 0$, then $\norm{\wtd{P}_0}_\tau\leq \Gamma m^k\sqrt{\eps}$.
\end{proposition}
\begin{proof}
     Suppose
    $m\notin \mcL$, $\eps\geq 0$, and $\varphi$ is an $(\eps,\mcX)$-synchronous state on $\C\Z_2^{*\mcX}\otimes \C\Z_2^{*\mcX}$ such that $\tau:=\varphi\circ\iota$ is an $(\eps,\mcR_m)$-state on $\C\Z_2^{*\mcX}$. By \Cref{thm:main}, part (a), $\wtd{P}_n$ is a contraction in $\C \Z_2^{*\mcX}$, $\wtd{X}_n$ is a monomial in $\C^*\ang{\mcX}$ of degree $4n+1$, and $\wtd{X}_n$ is a unitary involution in $\C\Z_2^{*\mcX}$, so \Cref{lemma:epstracial} implies
    \begin{align*}
        \left\vert\norm{\wtd{X}_n\wtd{P}_n\wtd{X}_n}_\tau -\norm{\wtd{P}_n}_\tau \right\vert=\left\vert\norm{\wtd{X}_n\wtd{P}_n\wtd{X}_n \otimes \Id}_\varphi -\norm{\wtd{P}_n \otimes \Id}_\varphi  \right\vert\leq (4n+1)\sqrt{\eps}
    \end{align*}
 for all $n\geq 0$. Hence
    \begin{align*}
        \left\vert\norm{\wtd{X}_n\wtd{P}_n\wtd{X}_n}^2_\tau -\norm{\wtd{P}_n}^2_\tau \right\vert &=\big(\norm{\wtd{X}_n\wtd{P}_n\wtd{X}_n}_\tau +\norm{\wtd{P}_n}_\tau \big)\left\vert\norm{\wtd{X}_n\wtd{P}_n\wtd{X}_n}_\tau -\norm{\wtd{P}_n}_\tau\right\vert\\
        &\leq \big( 2\norm{\wtd{P}_n}_\tau+ (4n+1)\sqrt{\eps}\big)(4n+1)\sqrt{\eps},
    \end{align*}
so
\begin{align}
\norm{\wtd{X}_n\wtd{P}_n\wtd{X}_n}^2_\tau \geq \norm{\wtd{P}_n}^2_\tau-(8n+2)\norm{\wtd{P}_n}_\tau\sqrt{\eps}-(4n+1)^2\eps\label{eq:squarediff}
\end{align}
for all $n\geq 0$. By \Cref{thm:main}, part (a) again, 
$\wtd{P}_n$ is a hermitian square in $\C\Z_2^{*\mcX}$ with $\norm{\wtd{P}_n}_{1}=1$ and $\norm{\wtd{P}_n}_{1,1}= 3(n+1)/2\leq 2n+2$ for every $n\geq 0$, so $\norm{\wtd{X}_n\wtd{P}_n\wtd{X}_n}_{1,1}\leq \norm{\wtd{P}_n}_{1,1}+ \norm{\wtd{P}_n}_{1}(8n+2) \leq 10n+4$.  \Cref{lemma:realpart} implies 
 \begin{align*}
\varphi\big(\wtd{X}_n\wtd{P}_n\wtd{X}_n\wtd{P}_n\otimes\Id  +\wtd{P}_n\wtd{X}_n\wtd{P}_n\wtd{X}_n\otimes\Id  \big)&=2\Re\Big(\varphi\big(\wtd{X}_n\wtd{P}_n\wtd{X}_n\wtd{P}_n\otimes\Id\big)\Big)\\
&\geq -2\norm{\wtd{X}_n\wtd{P}_n\wtd{X}_n}_{1,1}\norm{\wtd{P}_n\otimes\Id}_\varphi\sqrt{\eps}\\
&\geq -(20n+8)\norm{\wtd{P}_n}_{\tau}\sqrt{\eps}.
\end{align*}
By \Cref{eq:squarediff}, we obtain
 \begin{align*}
 \norm{\widetilde{P}_n+ \widetilde{X}_n\widetilde{P}_n\widetilde{X}_n}^2_\tau&=\norm{\widetilde{X}_n\widetilde{P}_n\widetilde{X}_n}_{\tau}^2+\norm{\widetilde{P}_n}_{\tau}^2+\varphi(\widetilde{X}_n\widetilde{P}_n\widetilde{X}_n\widetilde{P}_n\otimes\Id+\widetilde{P}_n\widetilde{X}_n\widetilde{P}_n\widetilde{X}_n\otimes\Id)\\
    &\geq2\norm{\widetilde{P}_n}_{\tau}^2-(28n+10)\norm{\widetilde{P}_n}_{\tau}\sqrt{\epsilon}-(4n+1)^2\epsilon\\
    &\geq  2\big(\norm{\widetilde{P}_n}_{\tau}-(7n+7)\sqrt{\eps}\big)^2 -(11n+11)^2\eps,
 \end{align*}
where the coefficients here (and below) are chosen for convenience. It follows that
 \begin{align*}
     2\big(\norm{\widetilde{P}_n}_{\tau}-(7n+7)\sqrt{\eps}\big)^2\leq \big(\norm{\widetilde{P}_n+ \widetilde{X}_n\widetilde{P}_n\widetilde{X}_n}_\tau+(11n+11)\sqrt{\eps}  \big)^2,
  \end{align*}    
and hence 
\begin{align*}
   \norm{\widetilde{P}_n}_{\tau}&\leq \frac{1}{\sqrt{2}}\big(\norm{\widetilde{P}_n+ \widetilde{X}_n\widetilde{P}_n\widetilde{X}_n}_\tau+(25n+25)\sqrt{\eps}  \big).  
\end{align*}
Let $C$ and $k$ be the constants from \Cref{thm:main}, part (b). By \Cref{prop:rbound}, $\norm{\widetilde{P}_n + \widetilde{X}_n \widetilde{P}_n
    \widetilde{X}_n-\widetilde{P}_{n+1}}_{\tau}\leq C((n+1)m)^k\sqrt{\eps}$ for all
    $n\geq 0$. Thus
    \begin{align*}
        \norm{\widetilde{P}_n}_{\tau}
        &\leq \tfrac{1}{\sqrt{2}}\big(\norm{\widetilde{P}_{n+1}}_{\tau}+C((n+1)m)^k\sqrt{\eps}+ (25n+25)\sqrt{\eps} \big)\\
        &\leq \tfrac{1}{\sqrt{2}}\norm{\widetilde{P}_{n+1}}_{\tau}+\tfrac{1}{\sqrt{2}}(C+25 
        )((n+1)m)^k\sqrt{\epsilon}.
    \end{align*}
    By induction, we obtain
    \begin{align*}
        \norm{\wtd{P}_0}_\tau\leq (\tfrac{1}{\sqrt{2}})^{t}\norm{\wtd{P}_{t}}_\tau+(C+25)\sum_{n=1}^{t}\frac{ n^k}{2^{n/2}} m^k\sqrt{\eps}
    \end{align*}
    for all $t\geq 1$. Since $\wtd{P}_t$ is a contraction, $\norm{\wtd{P}_t}_\tau\leq 1$ for all $t\geq 1$.
    As in the proof of \Cref{thm:maintracial}, the series $\sum_{n=1}^{\infty}\frac{ n^k}{2^{n/2}}$ converges, so taking
    $t\arr+\infty$, we conclude that 
    \begin{align*}
        \norm{\wtd{P}_0}_\tau\leq \Gamma m^k\sqrt{\eps}
    \end{align*}
    for any integer $\Gamma\geq (C+25)\sum_{n=1}^{\infty}\frac{ n^k}{2^{n/2}}$.
\end{proof}

\begin{proposition}\label{prop:actual}
Let $\mcW$ be the mapping from \Cref{lemma:tracialfree}. For every $m\in\N$, let $\mcW_m=\big(\bigcup_{r\in\mcR_m\cup\mcR_m^*}\mcW(r)\big)\cup\mcW(\wtd{P}_0)\cup\mcR_m$ be as in \Cref{prop:actaultracial}, and let 
\begin{align*}
    \wtd{\mcW}_m:=\{s\otimes \Id:s\in\mcW_m\}\cup\{s\otimes\Id,\Id\otimes s:s\in \bigcup_{x\in\mcX}\mcW(x) \}.
\end{align*}
There are positive integers $\wtd{\Gamma},k'$ such that if $m\notin\mcL$, $\varphi$ is an $(\eps,\wtd{\mcW}_m)$-state on $\C^*\ang{\mcX}\otimes\C^*\ang{\mcX}$, and $\varphi$ is $(\eps,\mcX)$-synchronous, then $\norm{\wtd{P}_0\otimes \Id}_\varphi\leq \wtd{\Gamma}m^{k'}\sqrt{\eps}$.
\end{proposition}
\begin{proof}
    Let $\Gamma,k$ be integers from \Cref{prop:productnonhalting}, and let $D$
    be the integer from \Cref{thm:main}, part (d). Suppose $m\not\in \mcL$,
    $\eps\geq 0$, and that $\varphi$ is an $(\eps,\wtd{\mcW}_m)$-state on
    $\C^*\ang{\mcX}\otimes\C^*\ang{\mcX}$ which is $(\eps,\mcX)$-synchronous. Since
    $\varphi(s^*s\otimes\Id)\leq \eps$ and $\varphi(\Id\otimes s^*s)\leq \eps$ for
    all $s\in \bigcup_{x\in\mcX}\mcW(x)$, and $\{s\otimes
    \Id:s\in\bigcup_{r\in\mcR_m\cup\mcR_m^*}\mcW(r)\}\subseteq\wtd{\mcW}_m$, by
    \Cref{lemma:free}, there is a $(25\eps,\mcX)$-synchronous state
    $\widehat{\varphi}$ on $\C\Z_2^{*\mcX}\otimes \C\Z_2^{*\mcX}$ such that 
    \begin{align*}
        \big\vert \norm{r\otimes\Id}_{\widehat{\varphi}}-\norm{r\otimes\Id}_\varphi\big\vert\leq 2\norm{r}_{1,1}\sqrt{\eps}\leq 2Dm\sqrt{\eps}
    \end{align*}
    for all $r\in \mcR_m\cup\mcR_m^*$. Since $\{r\otimes \Id:r\in \mcR_m\}\subseteq
    \wtd{\mcW}_m$, for every $r\in\mcR_m\cup\mcR_m^*$, $\norm{r\otimes
    1}_{\varphi}\leq \sqrt{\eps}$, and hence
    $\norm{r\otimes\Id}_{\widehat{\varphi}}\leq (2Dm+1)\sqrt{\eps}$. Rounding $2Dm+1$
    up to $2Dm+3$ so that $(2Dm+3)^2 \geq 25$, we conclude
    that $\widehat{\varphi}$ is a $((2Dm+3)^2\eps,\mcX)$-synchronous state on
    $\C\Z_2^{*\mcX}\otimes \C\Z_2^{*\mcX}$ and $\widehat{\varphi}\circ \iota$ is
    a $((2Dm+3)^2\eps,\mcR_m)$-state on $\C\Z_2^{*\mcX}$. 
    By \Cref{prop:productnonhalting},
    $\norm{\wtd{P}_0\otimes\Id}_{\widehat{\varphi}}\leq \Gamma m^k
    \sqrt{(2Dm+3)^2\eps}=(2Dm+3)\Gamma m^k \sqrt{\eps}$. Again by
    \Cref{lemma:free}, since $\{s\otimes 1:s\in\mcW(\wtd{P}_0)\}\subseteq
    \wtd{\mcW}_m$, we have 
     \begin{align*}
        \big\vert \norm{\wtd{P}_0\otimes\Id}_{\widehat{\varphi}}-\norm{\wtd{P}_0\otimes\Id}_\varphi\big\vert\leq 2\norm{\wtd{P}_0}_{1,1}\sqrt{\eps}\leq 3\sqrt{\eps}.
    \end{align*}
    It follows that 
    \begin{align*}
        \norm{\wtd{P}_0\otimes\Id}_\varphi\leq \left((2Dm+3)\Gamma m^k+3\right)\sqrt{\eps}\leq 8D\Gamma m^{k+1}\sqrt{\eps}.
    \end{align*}
    Taking $\wtd{\Gamma}=8D\Gamma$ and $k'=k+1$ completes the proof.
\end{proof}

Before proceeding to the proof of \Cref{thm:actualmain}, we need a few more
lemmas.  The first is a standard fact about extending tracial states to (not
necessarily tracial) states on tensor products.
\begin{lemma}\label{lem:extendtrace}
    If $\tau$ is a tracial state on $\C \Z_2^{*\mcX}$, then there is a $(0,\mcX)$-synchronous state
    $\varphi$ on $\C\Z_2^{*\mcX}\otimes \C\Z_2^{*\mcX}$ with $\varphi\circ \iota = \tau$. 
\end{lemma}
\begin{proof}
    Let $(\pi,v)$ be the GNS representation of $\tau$, and let $\mcC = \overline{\pi(\mcA)}$ be the closure of
    $\pi(\mcA)$ in the operator norm, where $\mcA = \C \Z_2^{*\mcX}$. By
    \cite[Proposition 6.1.2]{BO08}, the tracial state defined by $v$ on
    $\mcC$ extends to a state $\varphi' : \mcC \otimes \mcC^{op} \to \C : a \otimes b
    \mapsto \ang{v, ab v}$. Note that the isomorphism $\omega_\mcX : \C^*\ang{\mcX} \to \C^*\ang{\mcX}^{op}$ 
    defined in \Cref{sec:approxrepn} descends to an isomorphism $\mcA \to \mcA^{op}$, and that
    $\pi$ defines a $*$-homomorphism $\pi^{op} : \mcA^{op} \to \mcC^{op}$. Let $\varphi = 
    \varphi' \circ (\pi \otimes (\pi^{op} \circ \omega_{\mcX}))$. Then $\varphi$ is a state
    on $\mcA \otimes \mcA$ with $\varphi(a \otimes b) = \ang{v, \pi(a\omega_\mcX(b)) v} = 
    \tau(a \cdot \omega_\mcX(b))$. Since
    \begin{equation*}
        \norm{x \otimes 1 - 1 \otimes x}_{\varphi}^2 = 
            \varphi(x^* x \otimes 1 - x \otimes x^* - x^* \otimes x + 1 \otimes x^* x)
            = \tau(x^2 - x^2 - x^2 + x^2) = 0
    \end{equation*}
    for all $x \in \mcX$, $\varphi$ is $(0,\mcX)$-synchronous.
\end{proof}

The remaining two lemmas are used to get from $\Q \mcF_N$ and $\Q \Z_2^{*N}$
for large $N$ down to $\Q \Z_3^{*2}$ and $\Q \Z_2^{*3}$. Both lemmas will be familiar to experts.
\begin{lemma}\label{lem:subgroupsos}
    If $H$ be a subgroup of a group $G$, then any
    state on $\C H$ extends to a state on $\C G$. 
\end{lemma}

\begin{proof}
    Let $\C H_+$ and $\C G_+$ be the convex cones of sums of squares in $\C H$ and $\C G$ respectively, and let $r : \C G \to \C H$ be the linear function with $r(g) = g$ if $g \in H$
    and $r(g) = 0$ if $g \not\in H$. By \cite[Proposition 4]{SS11}, $r(\C G_+)\subseteq \C H_+$. Consequently, if $a \in \C G_+\cap\C H$, then $a = r(a) \in \C H_+$, so $\C G_+\cap \C H=\C H_+$. If $f$ is a state on $\C H$, then $f$ is nonnegative on $
    \C H_{+}$.  Since
    $\C G$ is archimedean, the M. Riesz extension theorem implies that $f$
    extends to a real linear functional on $\text{span}_{\R} \C G_{+}$ which is nonnegative
    on $\C G_+$. From there, we can extend further to get a linear functional
    $\wtd{f}$ on $\C G$ such that $\wtd{f}(a^*) = \overline{\wtd{f}(a)}$ for all $a
    \in \C G$, $\wtd{f}|_{\C H} = f$, and $\wtd{f}$ is nonnegative on $\C G_+$.
    Since $\wtd{f}(1) = f(1) = 1$, the archimedean condition implies that
    $\wtd{f}$ is a state. 
\end{proof}

Another way to prove this lemma is to observe that $\C H$ and $\C G$ are subalgebras of the full group $C^*$-algebras $C^*(H)$ and $C^*(G)$. By definition of the full group $C^*$-algebra, any state on $\C H$ is the restriction of a state on $C^*(H)$, and since $C^*(H)$ is a subalgebra of $C^*(G)$,  any state on $C^*( H)$ extends to a state on $C^*( G)$.

\begin{lemma}\label{lem:commute} 
    Suppose $N \geq 1$ and $G$ is either $\mcF_K$ or $\Z_3^{*K}$ with $K\geq 2$, or $\Z_2^{*K}$ with $K\geq 3$. Then there is an injective group
    homomorphism $\mu : \mcF_N \to G$ and a computable $*$-homomorphism
    $\wtd{\mu} : \Q^*\ang{x_1,\ldots,x_N} \to \Q^*\ang{x_1,\ldots,x_K}$
    such that the diagram
    \begin{equation*}
        \begin{tikzpicture}
            \node (1) at (0,0) {$\Q^*\ang{x_1,\ldots,x_N}$};
            \node (2) at (6,0) {$\Q^*\ang{x_1,\ldots,x_K}$}
                edge [<-] node [above] {$\wtd{\mu}$} (1);
            \node (3) at (0,-2) {$\Q \mcF_N$}
                edge [<-] node [left] {$q_1$} (1);
            \node (4) at (6,-2) {$\Q G$}
                edge [<-] node [below] {$\mu$} (3)
                edge [<-] node [right] {$q_2$} (2);
        \end{tikzpicture}
    \end{equation*}
    commutes, where $q_1$ and $q_2$ are the quotient homomorphisms sending
    $x_j$ to the $j$th generator of $\mcF_N$ and $G$ respectively.
\end{lemma}
\begin{proof}
    Let $x_1,\ldots,x_N$ denote the generators of $\mcF_N$, $y_1,y_2$ denote
    the generators of $\mcF_2$, $w_1, w_2$ denote the generators of
    $\Z_3^{*2}$, and $z_1,z_2,z_3$ denote the generators of 
    $\Z_2^{*3}$. It is well-known that the homomorphisms
    \begin{align*}
        &\gamma_1 : \mcF_N \to \mcF_2 : x_i \mapsto y_1^i y_2 y_1^{-i}, 1 \leq i \leq N, \\
        &\gamma_2 : \mcF_2 \to \Z_3^{*2} : y_1 \mapsto w_1 w_2^{-1}, y_2 \mapsto w_1^{-1} w_2, \text{ and} \\
        &\gamma_3 : \mcF_2 \to \Z_2^{*3} : y_i \mapsto z_1^{1-i} z_2 z_3 z_1^{1-i}, i=1,2
    \end{align*}
    are injective. If $G = \mcF_K$ with $K \geq 2$, then we can take $\mu = \gamma_1$, where
    we think of $\mcF_2$ as a subgroup of $\mcF_K$ by mapping $y_1$ and $y_2$ to the first
    two generators, and let $\wtd{\mu}$ be the lift of $\mu$
    sending $x_i \mapsto x_1^i x_2 (x_1^*)^i$. 
    If $G = \Z_3^{*K}$ with $K \geq 2$, we can take $\mu = \gamma_2 \circ \gamma_1$, where
    again we think of $\Z_3^{*2}$ as a subgroup of $\Z_3^{*K}$ by mapping $w_1$ and $w_2$
    to the first two generators, and let $\wtd{\mu}$ be the lift of $\mu$
    sending $x_i \mapsto (w_1 w_2^*)^i w_1^* w_2 (w_2 w_1^*)^i$. 
    Finally if $G = \Z_2^{*K}$ with $K \geq 3$, we can take $\mu = \gamma_3
    \circ \gamma_1$, where we think of $\Z_2^{*3}$ as a subgroup of
    $\Z_2^{*K}$, and let $\wtd{\mu}$ be the lift of $\mu$ sending $x_i \mapsto
    (x_2 x_3)^i (x_1 x_2 x_3 x_1) (x_3^* x_2^*)^i$. In all three cases, we see from
    the explicit formula that $\wtd{\mu}$ is computable.
\end{proof}

\begin{proof}[Proof of Theorem \ref{thm:actualmain}] 
Set $\mcP_N:=\Q^*\ang{x_1,\ldots,x_N}$ for $N\in\N$.  Let $P_1(N)$ be the
promise problem from the statement of the theorem with $\mcP=\mcP_N$ and $\mcC
= \Q \Z_2^{*N}$, and let $P_2(N)$ be the promise problem with $\mcP=\mcP_N$ and
$\mcC = \Q \Z_3^{*N}$. Although it isn't mentioned in the theorem statement,
let $P_3(N)$ be the promise problem with $\mcP=\mcP_N$ and $\mcC = \Q \mcF_N$. 
As in the proof of \Cref{thm:actualmaintracial}, let $n_0=\abs{\mcX}$.  We
start by showing that $P_1(N)$ is coRE-hard for all $N \geq n_0$. 

For this, suppose $N\geq n_0$, and once again pick an ordering on $\mcX$ and identify $\Q^*\ang{\mcX}$, $\C^*\ang{\mcX}$, and $\C\Z_2^{*\mcX}$ with subalgebras of $\mcP_N$, $\C^*\ang{x_1,\ldots,x_N}$, and $\C\Z_2^{*N}$ respectively. Let $\pi:\mcP_N \to\Q\Z_2^{*N}$ be the quotient map, and let $q=q_2\circ q_1$, where 
$q_1:\Q\Z_2^{*N}\to\Q\Z_2^{*\mcX}$ is the homomorphism sending $x_i\mapsto x_i$ for $1\leq i \leq n_0$, and $q_2:\Q\Z_2^{*\mcX}\to \mcA_\mcL(m)$ is the quotient map. Let
$\wtd{\Gamma},k'$, and $\wtd{W}_m$ be as in \Cref{prop:actual}.  For every
$m\in\N$, define a $*$-polynomial 
\begin{align*}
\beta(m)=\sum_{r\in\wtd{\mcW}_m\cup\wtd{\mcW}_m^*}r^*r +\sum_{x\in\mcX} (x\otimes \Id-\Id\otimes x)^*(x\otimes \Id-\Id\otimes x)-\tfrac{1}{\wtd{\Gamma}^2m^{2k'}}\wtd{P}_0^*\wtd{P}_0\otimes \Id  
\end{align*}
in $\Q^*\ang{\mcX}\otimes_\Q\Q^*\ang{\mcX}\subseteq\mcP_N\otimes_\Q\mcP_N$. Then $\beta(m)$ is self-adjoint and the map $\beta:\N\to\mcP_N\otimes_\Q\mcP_N$ is computable by \Cref{lemma:tracialfree,thm:main,prop:actual}. 

By \Cref{lemma:tracialfree}, $(\pi\otimes\pi)(t\otimes \Id)=(\pi\otimes\pi)(s\otimes \Id)=(\pi\otimes\pi)(\Id\otimes s)=0$ for all $t\in \big(\bigcup_{r\in\mcR_m\cup\mcR_m^*}\mcW(r)\big)\cup\mcW(\wtd{P}_0)$ and $s\in\bigcup_{x\in\mcX}\mcW(x)$. Hence 
\begin{align*}
    (\pi\otimes\pi)(\beta(m))=\sum_{r\in\mcR_m\cup\mcR_m^*}r^*r\otimes \Id + \sum_{x\in\mcX} (x\otimes \Id-\Id\otimes x)^2-\tfrac{1}{\wtd{\Gamma}^2m^{2k'}}\wtd{P}_0^2\otimes \Id
\end{align*}
in $\Q\Z_2^{*\mcX}\otimes_\Q \Q\Z_2^{*\mcX}\subseteq \Q\Z_2^{*N}\otimes_\Q \Q\Z_2^{*N}$. If $m\in\mcL$, then by \Cref{thm:main}, part (c), there is a tracial state $\tau$ on $\mcA_\mcL(m)$ such that $\tau(\wtd{P}_0^2)>0$. Since $\tau\circ q_2$ is a tracial state on $\C\Z_2^{*\mcX}$ and $q_2(r)=0$ for all $r\in \mcR_m$, \Cref{lem:extendtrace} implies that there is a $(0,\mcX)$-synchronous state $\varphi$ on $\C\Z_2^{*\mcX}\otimes\C\Z_2^{*\mcX}$ such that $\varphi(\wtd{P}_0^2\otimes \Id)=\tau(\wtd{P}_0^2)>0$ and $\varphi(r^*r\otimes \Id)=\tau(q_2(r^*r))=0$ for all $r\in\mcR_m\cup\mcR_m^*$. Then $\wtd{\varphi}:=\varphi\circ(q_1\otimes q_1)$ is  a state on $\C\Z_2^{*N}\otimes\C\Z_2^{*N}$ such that $\wtd{\varphi}\left( (\pi\otimes\pi)(\beta(m))\right)=-\tfrac{1}{\wtd{\Gamma}^2m^{2k'}}\tau(\wtd{P}_0^2)<0$. We conclude that $(\pi\otimes\pi)(\beta(m))$ is not positive. 

Now suppose $m\notin\mcL$. Let $\varphi$ be a state on $(\mcP_N)_\C\otimes_\C(\mcP_N)_\C$, and let
\begin{align*}
\eps:=\varphi\left(\sum_{r\in\wtd{\mcW}_m\cup\wtd{\mcW}_m^*}r^*r +\sum_{x\in\mcX} (x\otimes \Id-\Id\otimes x)^*(x\otimes \Id-\Id\otimes x) \right).
\end{align*} 
The restriction of $\varphi$ to $\C^*\ang{\mcX}\otimes \C^*\ang{\mcX}$ is an
$(\eps,\wtd{\mcW}_m)$-state which is $(\eps,\mcX)$-synchronous, so by
\Cref{prop:actual},
$\varphi(\wtd{P}_0^*\wtd{P}_0\otimes\Id)=\norm{\wtd{P}_0\otimes\Id}^2_\varphi\leq
\wtd{\Gamma}^2m^{2k'}\eps$. Hence $\varphi (\beta(m))\geq 0$. We conclude that the complement $\overline{\mcL}$ of $\mcL$
reduces to $P_1(N)$. Taking $\mcL$ to
be RE-hard implies that $P_1(N)$ is coRE-hard for $N \geq n_0$.

Next, the quotient map $\pi^{(N)}: \mcP_N\to\Q\Z_2^{*N}$ factors as
$\pi^{(N)}=\pi_2^{(N)}\circ\pi_1^{(N)}$, where $\pi_1^{(N)}$ is the quotient
map $\mcP_N \to \Q \F_N$, and $\pi_2^{(N)}$ is the quotient map 
$\Q \F_N \to \Q \Z_2^{*N}$. If $(\pi \otimes \pi)(\alpha)$ is not positive in $\C
\Z_2^{*N} \otimes \C \Z_2^{*N}$ for $\alpha \in \mcP_N \otimes \mcP_N$, then
$(\pi_1^{(N)} \otimes \pi_1^{(N)})(\alpha)$ is not positive in $\C \F_N \otimes \C \F_N$. Thus
$P_1(N)$ reduces to $P_3(N)$, and $P_3(N)$ is coRE-hard
for $N \geq n_0$. 

By Lemma \ref{lem:commute}, if $N \geq 1$ and $K \geq 2$ then there is an
injective homomorphism $\mu : \mcF_{N} \to \Z_3^{*K}$ and a computable
homomorphism $\wtd{\mu} : \mcP_{N} \to \mcP_{K}$ such that $\pi_3^{(K)} \circ
\wtd{\mu} = \mu \circ \pi_1^{(N)}$, where $\pi_3^{(K)}$ is the quotient map
$\mcP_K \to \Q \Z_3^{*K}$. If $\alpha \in \mcP_N \otimes \mcP_N$ is
positive in $(\mcP_N)_{\C} \otimes (\mcP_N)_{\C}$, then $(\wtd{\mu} \otimes
\wtd{\mu})(\alpha)$ is positive in $(\mcP_K)_{\C} \otimes (\mcP_K)_{\C}$. If
$(\pi_1^{(N)} \otimes \pi_1^{(N)})(\alpha)$ is not positive in $\C \F_N \otimes
\C \F_N$, then there is a state $f$ on $\C \F_N \times \F_N$ such that
$f(\pi_1^{(N)} \otimes \pi_1^{(N)})(\alpha)) < 0$.  If we identify $\C \mcF_N
\otimes \C \mcF_N$ with $\C \mcF_N \times \mcF_N$ and $\C \Z_3^{*K} \otimes \C
\Z_3^{*K}$ with $\C \Z_3^{*K} \times \Z_3^{*K}$, then the homomorphism $\mu \otimes \mu$
is the algebra homomorphism coming from the injective group homomorphism $\mu
\times \mu : \mcF_N \times \mcF_N \to \Z_3^{*K} \times \Z_3^{*K}$.  By
\Cref{lem:subgroupsos}, the state $f$ extends to a state $\wtd{f}$ on $\C
\Z_3^{*K} \times \Z_3^{*K}$, and 
\begin{equation*} 
    \wtd{f}((\mu \times \mu)(\pi_1^{(N)} \otimes \pi_1^{(N)})(\alpha)) = f((\pi_1^{(N)} \otimes \pi_1^{(N)})(\alpha)) < 0.
\end{equation*}
Hence $(\pi_3^{(K)} \otimes \pi_3^{(K)})(\wtd{\mu} \otimes \wtd{\mu})(\alpha) =
(\mu \otimes \mu)(\pi_1^{(N)} \otimes \pi_1^{(N)})(\alpha)$ is not
positive.  Since $\wtd{\mu} \otimes \wtd{\mu}$ is computable, 
$P_3(N)$ reduces to $P_2(K)$ for any $N \geq 1$ and $K \geq 2$. It follows that
$P_2(K)$ is coRE-hard for all $K \geq 2$. Similarly, $P_3(N)$ reduces to
$P_1(K)$ for all $N \geq 1$ and $K \geq 3$, so $P_1(K)$ is coRE-hard for all $K
\geq 3$.
\end{proof}

The following example shows that \Cref{thm:actualmain} is not true for
$\mcC = \Q \Z_2^{*2}$. 
\begin{example}\label{ex:Z22}
    Let $G = \Z_2^{*2}$ be the infinite dihedral group. This group is amenable
    and linear. Hence the product group $G \times G$ is also amenable and linear.
    By a theorem of Bekka, $\C G \times G$ is residually finite-dimensional \cite{Bekka}, in
    the sense that an element $\alpha \in \C G \otimes \C G = \C G \times G$ is
    positive if and only if $\phi(\alpha)$ is positive for all
    finite-dimensional representations $\phi$.  Every finite-dimensional
    representation of $\C G \times G$ decomposes into irreducible
    representations, so $\alpha$ is positive if and only if $\phi(\alpha)$ is
    positive for all irreducible finite-dimensional representations. The
    irreducible representations of $G$ have dimension $1$ and $2$, so the
    irreducible representations of $G \times G$ have dimension $1$, $2$, and
    $4$. Thus $\alpha \in \C G \times G$ is positive if and only if $\phi(\alpha)$
    is positive for all representations of dimension $4$. The question of
    whether $\phi(\alpha)$ is not positive for all representations of dimension
    $4$ lies in the existential theory of the reals, and thus can be solved in
    $\PSPACE$ \cite{Can88}.  In particular, the question of whether $\alpha \in
    \C G \times G$ is positive is decidable. We do not know whether 
    $\C G \times G$ is archimedean closed. 
\end{example}

\bibliographystyle{alpha}
\bibliography{bibliography}

\end{document}

%% file: ams_header.tex
\usepackage{amsmath,amsthm,amsfonts,amssymb,latexsym,verbatim,bbm,hyperref}
\usepackage{mathtools}
\usepackage[shortlabels]{enumitem}
\usepackage[top=1.5in,bottom=1.5in,left=1.5in,right=1.5in]{geometry}
\usepackage{thmtools}
\usepackage{subcaption}
\usepackage[capitalize]{cleveref}
\usepackage{afterpage}

\DeclareRobustCommand{\SkipTocEntry}[5]{}

\allowdisplaybreaks

\setlist{itemsep=.5\baselineskip,topsep=.5\baselineskip}

\numberwithin{equation}{section}
\theoremstyle{plain}
\newtheorem{theorem}{Theorem}[section]

\newtheorem{lemma}[theorem]{Lemma}
\newtheorem{proposition}[theorem]{Proposition}

\newtheorem{example}[theorem]{Example}
\newtheorem{definition}[theorem]{Definition}

\newtheorem{remark}[theorem]{Remark}

\newtheorem{cor}[theorem]{Corollary}

\newtheorem{conjecture}[theorem]{Conjecture}

\newcommand{\iso}{\cong}
\newcommand{\arr}{\rightarrow}
\newcommand{\incl}{\hookrightarrow}

\newcommand{\R}{\mathbb{R}}
\newcommand{\C}{\mathbb{C}}
\newcommand{\Z}{\mathbb{Z}}
\newcommand{\N}{\mathbb{N}}
\newcommand{\Q}{\mathbb{Q}}
\newcommand{\K}{\mathbb{K}}

\newcommand{\eps}{\epsilon}

\newcommand{\Id}{1}

\newcommand{\mcA}{\mathcal{A}}
\newcommand{\mcB}{\mathcal{B}}
\newcommand{\mcC}{\mathcal{C}}

\newcommand{\mcF}{\mathcal{F}}

\newcommand{\mcH}{\mathcal{H}}

\newcommand{\mcL}{\mathcal{L}}

\newcommand{\mcP}{\mathcal{P}}

\newcommand{\mcR}{\mathcal{R}}

\newcommand{\mcW}{\mathcal{W}}
\newcommand{\mcX}{\mathcal{X}}

\newcommand{\mbN}{\mathbb{N}}

\newcommand{\abs}[1]{\lvert\tinyspace #1 \tinyspace\rvert}
\renewcommand{\Re}{\operatorname{Re}}